\documentclass[12pt]{article}
\usepackage{amsmath,amssymb,amsthm}

\newtheorem{thm}{Theorem}
\newtheorem{crl}{Corollary}
\newtheorem{cnj}{Conjecture}
\newtheorem{prp}{Proposition}
\newtheorem{lmm}{Lemma}
\newtheorem{rmk}{Remark}
\newtheorem{dfn}{Definition}
\newtheorem{exm}{Example}

\makeatletter
\newcommand{\subjclass}[2][2010]{%
  \let\@oldtitle\@title%
  \gdef\@title{\@oldtitle\footnotetext{#1 \emph{Mathematics subject classification(s).} #2}}%
}
\newcommand{\keywords}[1]{%
  \let\@@oldtitle\@title%
  \gdef\@title{\@@oldtitle\footnotetext{\emph{Key words and phrases.} #1.}}%
}
\makeatother

\title{On the algebraicity of some products of special values of Barnes' multiple gamma function}
\author{Tomokazu Kashio\thanks{Tokyo University of Science, \texttt{kashio\_tomokazu@ma.noda.tus.ac.jp}}}

\begin{document}
\maketitle

\begin{abstract}
We consider partial zeta functions $\zeta(s,c)$ associated with ray classes $c$'s of a totally real field.
Stark's conjecture implies that an appropriate product of $\exp(\zeta'(0,c))$'s is an algebraic number which is called a Stark unit.
Shintani gave an explicit formula for $\exp(\zeta'(0,c))$ in terms of Barnes' multiple gamma function.
Yoshida ``decomposed'' Shintani's formula: he defined the symbol $X(c,\iota)$ satisfying that $\exp(\zeta'(0,c))=\prod_{\iota} \exp(X(c,\iota))$
where $\iota$ runs over all real embeddings of $F$.
Hence we can decompose a Stark unit into a product of $[F:\mathbb Q]$ terms.
The main result is to show that $([F:\mathbb Q]-1)$ of them are algebraic numbers.
We also study a relation between Yoshida's conjecture on CM-periods and Stark's conjecture. 
\end{abstract}

\section{Introduction.}

We recall the rank $1$ abelian Stark conjecture.
Let $K/F$ be an abelian extension of number fields, and $S$ a finite set of places of $F$.
In this paper, we always assume that 
\begin{center}
$S$ contains all primes ramifying in $K/F$ and all infinite places of $F$, and $|S| \geq 3$.
\end{center}
We put $G:=\mathrm{Gal}(K/F)$.
The partial zeta function $\zeta_S(s,\tau)$ associated with $\tau \in G$ is defined by
\begin{equation*}
\zeta_S(s,\tau):=\sum_{\mathfrak a \subset \mathcal O_F,\, (\mathfrak a,S)=1,\ (\frac{K/F}{\mathfrak a})=\tau} 
N\mathfrak a^{-s}.
\end{equation*}
Here $\mathfrak a$ runs over all integral ideals of $F$, relatively prime to any prime ideal in $S$, 
whose image under the Artin map $(\frac{K/F}{*})$ is equal to $\tau$.
This series converges for $\mathrm{Re}(s)>1$, has a meromorphic continuation to $\mathbb C$, which is analytic at $s=0$.

\begin{cnj}[The rank $1$ abelian Stark conjecture] \label{sc}
Let $F,K,S,G$ be as above.
We denote the number of roots of unity in $K$ by $e_K$.
Assume that there exists a place $v$ in $S$ splitting completely in $K$, and fix a place $w$ of $K$ lying above $v$. 
Then there exists an element $\epsilon \in K^\times$, which is called a Stark unit, satisfying 
\begin{enumerate}
\item $\epsilon$ is a $v$-unit.
\item $\log |\tau(\epsilon)|_w=-e_K \zeta_S'(0,\tau)$ for all $\tau \in G$. 
\item $K(\epsilon^{1/e_K})/F$ is an abelian extension.
\end{enumerate}
\end{cnj}

We note that the statement of the rank $1$ abelian Stark conjecture in the case $|S|=2$ is slightly different from the above, and 
has been proved (for the proof, see \cite[Proposition 4.3.11]{Da}).

In the following, we only consider the case when the splitting place $v$ is a real infinite place, and focus on the property that $\epsilon$ is a unit.
We note that 
\begin{itemize}
\item If more than one place in $S$ splits completely in $K$, then $\zeta_S'(0,\tau)=0$ for all $\tau \in G$, so Conjecture \ref{sc} is trivial.
Therefore, when $v$ is a real infinite place, we may assume that $F$ is totally real. (Note that a complex infinite place always splits completely.)
\item There exists a real infinite place of $F$ splitting completely if and only if $K$ is not totally complex.
\item We can write a Stark unit explicitly as $\epsilon=\exp(-e_K \zeta_S'(0,\mathrm{id}))$ when $w$ is a real infinite place.
\end{itemize}
Therefore our problem can be formulated as follows:

\begin{dfn}
Let $\mathbb R_+$ be the multiplicative group of positive real numbers.
For any subgroup $X$ of $\mathbb R_+$, 
we put $X^\mathbb Q:=\{x \in \mathbb R_+ \mid $ there exists $n \in \mathbb N$ satisfying $x^n \in X\}$.
\end{dfn}

\begin{cnj}[A part of Stark's conjecture] \label{wsc}
Let $K$ be a finite abelian extension of a totally real field $F$ and $G:=\mathrm{Gal}(K/F)$.
Assume that $K$ is not totally complex. We fix a real embedding of $K$
and regard $F,K$ as subfields of $\mathbb R$ (i.e., $F \subset K \subset \mathbb R$).
Then we have  
\begin{equation} \label{fwsc}
\exp (\zeta_S'(0,\tau)) \in \left(\mathcal O_K^\times \cap \mathbb R_+ \right)^\mathbb Q \qquad(\tau \in G).
\end{equation}
Here we denote by $\mathcal O_K^\times$ the group of units in $K$.
\end{cnj}

Let $K$ be a finite abelian extension of a totally real field $F$, and $S,\tau \in G$ as above.
In particular, we do not assume that $K$ is not totally complex.
In \cite{Yo}, 
Yoshida decomposed $\exp (\zeta_S'(0,\tau))$ into the following form.
For an integral divisor $\mathfrak f$ of $F$, we denote by $C_\mathfrak f$ the ray class group modulo $\mathfrak f$.
Let $\mathfrak f_{K/F}$ be the conductor of $K/F$ and $\mathrm{Art} \colon C_{\mathfrak f_{K/F}} \rightarrow G$ the Artin map.
For an integral divisor $\mathfrak f$ with $\mathfrak f_{K/F} | \mathfrak f$, 
we denote the composite map $C_\mathfrak f \rightarrow C_{\mathfrak f_{K/F}} \stackrel{\mathrm{Art}}\rightarrow G$ by $\mathrm{Art}_\mathfrak f$.
Then we obtain (Theorem \ref{candec})
\begin{equation*}
\exp (\zeta_S'(0,\tau)) =\prod_{\iota \colon F \hookrightarrow \mathbb R} \prod_{c \in \mathrm{Art}_{\mathfrak f_S}^{-1}(\tau)} \exp(X(c,\iota)),
\end{equation*}
if we take 
\begin{equation} \label{Stof}
\mathfrak f_S:=\mathfrak f_{K/F} \times \text{``the product of all places $\lambda \in S$ with $\lambda \nmid \mathfrak f_{K/F}$''}.
\end{equation}
Here $\iota$ runs over all real embeddings of $F$, 
and $X(c,\iota)$ is Yoshida's class invariant defined in \S \ref{dfnofX}, in terms of Barnes' multiple gamma function and Shintani's cone decomposition.
The main result (Theorem \ref{main}) in this paper states that under the assumption of Conjecture \ref{wsc}, 
there exist a totally positive unit $\epsilon \in \mathcal O_F^\times$ and a natural number $m$ satisfying 
\begin{center}
$\displaystyle \prod_{c \in \mathrm{Art}_{\mathfrak f_S}^{-1}(\tau)} \exp(X(c,\iota)) =\iota(\epsilon)^\frac{1}{m}$ whenever $\iota\neq \mathrm{id}$.
\end{center}
As a direct application, we can refine the statement of Conjecture \ref{wsc} as follows:
\begin{equation*} 
\prod_{c \in \mathrm{Art}_{\mathfrak f_S}^{-1}(\tau)} \exp(X(c,\mathrm{id})) \in \left(\mathcal O_K^\times \cap \mathbb R_+ \right)^\mathbb Q.
\end{equation*}
Moreover, the main result in this paper has the following significance:
In \S 5, we will discuss a relation between 
the algebraicity of Stark units, monomial relations among CM-periods,
and Yoshida's conjecture (Conjecture \ref{YC}) on Shimura's period symbol.
Theorem \ref{main} and its Corollary are necessary for this discussion (to be precise, for the proofs of
Propositions \ref{trans}, \ref{mainprp}, \ref{mainprp2}).

Let us explain the outline of this paper. 
In \S 2, we introduce Yoshida's technique for decomposing $\zeta_S'(0,\tau)$:
Associated with a ray class $c$ and a real embedding $\iota$ of $F$, 
we define the invariant $X(c,\iota)$ in terms of Barnes' multiple gamma function and Shintani's cone decomposition.
Then we state a modified version of Shintani's formula (Theorem \ref{candec}) which gives the canonical decomposition of $\zeta_S'(0,\tau)$.
In \S 3, we state our main result (Theorem \ref{main}) and provide the proof.
When $[F:\mathbb Q]=2$, this result is due to Yoshida, and our proof is a generalization of his proof.
One of the new ideas of this paper is to define ``formal multiple zeta values'' in (\ref{amzv}), 
which we need since we treat the case when usual multiple zeta function does not converges (e.g., in Definition \ref{defofvab}).
In \S 4, we give the proof of Lemma \ref{keylemma} which states that, roughly speaking, we can take a suitable cone decomposition for computing $X(c,\iota)$'s.
This Lemma is a key step in the proof of Theorem \ref{main}.
As an application, in \S 5, we study a relation between Yoshida's conjecture on CM-periods and Stark's conjecture:
Yoshida formulated a Conjecture (Conjecture \ref{YC}) which expresses Shimura's period symbol in terms of $\exp(X(c,\iota))$'s.
By using Theorem \ref{main}, we can reformulate this Conjecture to the form (\ref{equiv}).
We note that (\ref{equiv}) has a natural generalization (Conjecture \ref{maincnj}) which also implies a part of Stark's conjecture.
Furthermore, we give some evidence for Conjecture \ref{maincnj}.

\section{Yoshida's \mbox{\boldmath $X$}-invariant.} \label{dfnofX}

In this section, we introduce Yoshida's $X$-invariant, which is defined as the sum of three invariants $G,W,V$.
Barnes' multiple zeta function $\zeta(s,\mathbf a,z)$ for $z\in \mathbb R_+$, $\mathbf a=(a_1,\dots,a_r) \in \mathbb R_+^r$ is defined by
\begin{equation*} 
\zeta(s,\mathbf a,z):=\sum_{0\leq m_1,m_2,\dots,m_r \in \mathbb Z}(z+m_1a_1+\dots+m_ra_r)^{-s}.
\end{equation*}
This series converges for $\mathrm{Re}(s)>r$, has a meromorphic continuation to $\mathbb C$, which is analytic at $s=0$.

\begin{dfn}
We define the multiple gamma function as
\begin{equation*}
\Gamma(z,\mathbf a):=\exp (\zeta'(0,\mathbf a,z)) \qquad(z\in \mathbb R_+,\ \mathbf a \in \mathbb R_+^r).
\end{equation*}
We note that this definition is slightly different from 
Barnes' definition:
$\exp (\zeta'(0,\mathbf a,z)) =\Gamma_r(z,\mathbf a)/\rho_r(\mathbf a)$ with a correction term $\rho_r(\mathbf a)$.
\end{dfn}

Let $F$ be a totally real field of degree $n$, $\mathcal O$ its ring of integers, and $\infty_1,\dots,\infty_n$ all the infinite places of $F$.
We denote the set of all totally positive elements in $F,\mathcal O,\mathcal O^\times$ by $F_+,\mathcal O_+,\mathcal O^\times_+$ respectively.
Let $\mathfrak f$ be an integral divisor of $F$ of the form 
\begin{equation*}
\mathfrak f=\mathfrak m\infty_1\dotsm\infty_n
\end{equation*}
with $\mathfrak m$ an integral ideal.
We denote the ray class group modulo $\mathfrak f$ by $C_\mathfrak f$.
The partial zeta function $\zeta(s,c)$ associated with $c\in C_\mathfrak f$ is defined by
\begin{equation*}
\zeta(s,c):=\sum_{\mathfrak a\subset \mathcal O,\, \mathfrak a \in c}N \mathfrak a ^{-s},
\end{equation*}
where $\mathfrak a$ runs over all integral ideals in the ray class $c$.
This series also has a meromorphic continuation to $\mathbb C$, which is analytic at $s=0$.
When $K/F$ is an abelian extension, taking $S,\mathfrak f_S$ as in (\ref{Stof}), 
we can write for $\tau\in \mathrm{Gal}(K/F)$
\begin{equation*}
\zeta_S(s,\tau)=\sum_{c\in \mathrm{Art}_{\mathfrak f_S}^{-1}(\tau)}\zeta(s,c).
\end{equation*}

We identify $F\otimes \mathbb R = \mathbb R^n$ by $z\otimes \alpha \mapsto (\alpha\iota_k(z))_{k}$,
where $\iota_1,\dots,\iota_n$ are all real embeddings of $F$.
We consider a ``cone decomposition'' of a subset of $F\otimes \mathbb R$ in the following sense.

\begin{dfn}
Assume that $v_1,\dots,v_r\in \mathcal O$ are linearly independent in $F\otimes \mathbb R$ over $\mathbb R$.
We define an ($r$-dimensional open simplicial) cone $C(\mathbf v)$ with the basis $\mathbf v:=(v_1\dots,v_r)$ by
\begin{equation*}
C(\mathbf v)=C(v_1,\dots,v_r):=\left\{\sum_{i=1}^r x_iv_i \in F\otimes \mathbb R \mid x_1,x_2,\dots,x_r \in \mathbb R_+ \right\}.
\end{equation*}
We note that whenever we consider a cone $C(v_1,\dots,v_r)$ in this paper, we assume that $v_1,\dots,v_r$ are in $\mathcal O$ and linearly independent.
\end{dfn}

We put $F\otimes \mathbb R_{n+}:=\mathbb R_+^n=\{(x_1,\dots,x_n) \in \mathbb R^n \mid x_1,\dots,x_n>0\}$.
Let $D$ be a subset of $F\otimes \mathbb R$.
Throughout this section, we always assume that 
\begin{itemize}
\item $D$ has a cone decomposition of the form $D=\coprod_{j \in J}C(\mathbf v_j)$, 
where the symbol $\coprod$ denotes the disjoint union, 
$J$ is a finite set of indices, 
and $C(\mathbf v_j)$ is a cone with the basis $\mathbf v_j=(v_{j1},\dots,v_{jr(j)}) \in \mathcal O^{r(j)}$ ($j \in J$, $1 \leq r(j) \leq n$).
\item $D\subset F\otimes \mathbb R_{n+}$. Namely, all of the above $v_{ji}$'s are totally positive.
\end{itemize}
We note that we will relax the second condition for $D$ from the next section.

For each $c \in C_\mathfrak f$, we fix an integral ideal ${\mathfrak a}_c$ satisfying that 
${\mathfrak a}_c\mathfrak f$ and $c$ belong to the same narrow ideal class (in $C_{(1)\infty_1\dotsm\infty_n}$).
For $j \in J$ we put
\begin{equation*}
R(c,{\mathfrak a}_c,\mathbf v_j)
:=\left \{\mathbf x=(x_1,\dots,x_{r(j)}) \in (\mathbb Q \cap (0,1])^{r(j)} \mid z:=\sum_{i=1}^{r(j)} x_i v_{ji} \in ({\mathfrak a}_c \mathfrak f)^{-1},\, 
z{\mathfrak a}_c \mathfrak f \in c\right\}.
\end{equation*}
We see that $R(c,{\mathfrak a}_c,\mathbf v_j)$ is a finite set (or the empty set)
since the condition $\sum_{i=1}^{r(j)} x_i v_{ji} \in ({\mathfrak a}_c \mathfrak f)^{-1}$ implies that the denominators of $x_i$ are bounded.
Moreover, we can write \cite[Chapter II, Lemma 3.2]{Yo}
\begin{equation} \label{basic}
\begin{split}
&\left \{z \in ({\mathfrak a}_c \mathfrak f)^{-1} \cap C(\mathbf v_j),\, z{\mathfrak a}_c \mathfrak f \in c\right\} \\
&=\coprod_{\mathbf x \in R(c,{\mathfrak a}_c,\mathbf v_j)}\left\{(x_1+m_1)v_{j1}+\dots + (x_{r(j)}+m_{r(j)})v_{jr(j)} \mid 0\leq m_1,\dots,m_{r(j)} \in \mathbb Z\right\}.
\end{split}
\end{equation}
Hence for any real embedding $\iota$ of $F$, we obtain
\begin{equation*}
\sum_{z \in ({\mathfrak a}_c \mathfrak f)^{-1}\cap D,\, (z){\mathfrak a}_c \mathfrak f \in c }\iota(z)^{-s}
=\sum_{j \in J}\sum_{\mathbf x \in R(c,{\mathfrak a}_c,\mathbf v_j)}\zeta(s,\iota(\mathbf v_j),\iota(\mathbf x\,{}^t\mathbf v_j)).
\end{equation*}
In particular, this series has a meromorphic continuation to $\mathbb C$, which is analytic at $s=0$.
Then the $G$-invariant \cite[Chapter III, (3.6), (3.28)]{Yo} is defined by
\begin{equation*}
G(c,\iota;D,{\mathfrak a}_c):=\left[\frac{d}{ds}\sum_{z \in ({\mathfrak a}_c \mathfrak f)^{-1}\cap D,\, (z){\mathfrak a}_c \mathfrak f \in c }\iota(z)^{-s} \right]_{s=0}
=\sum_{j \in J}\sum_{\mathbf x \in R(c,{\mathfrak a}_c,\mathbf v_j)}\log \Gamma(\iota(\mathbf x\,{}^t\mathbf v_j),\iota(\mathbf v_j)).
\end{equation*}
In \cite[Chapter III, (3.8), (3.31)]{Yo}, the $W$-invariant was defined by
\begin{equation} \label{worig}
W(c,\iota;D,{\mathfrak a}_c):=-\frac{1}{n}\log N{\mathfrak a}_c \mathfrak f \cdot \left[\sum_{z \in ({\mathfrak a}_c \mathfrak f)^{-1}\cap D,\, 
(z){\mathfrak a}_c \mathfrak f \in c }\iota(z)^{-s} \right]_{s=0}.
\end{equation}
In the present paper, we slightly modify this definition.
Let $\mathrm{FI}_F$ be the group of all fractional ideals of $F$.
We define a group homomorphism $\log_\iota \colon \mathrm{FI}_F \rightarrow \mathbb R$ for each real embedding $\iota$ of $F$ in the following manner.
For each prime ideal $\mathfrak p$, we choose $\pi_\mathfrak p \in \mathcal O_+$ satisfying $\mathfrak p^{h_F^+}=(\pi_\mathfrak p)$,
where $h_F^+$ is the narrow class number.
Then we put $\log_\iota \mathfrak p:=\frac{1}{h_F}\log \iota(\pi_\mathfrak p)$ for every $\iota$, and extend this linearly 
to $\log_\iota \colon \mathrm{FI}_F \rightarrow \mathbb R$. 
We easily see the following (\ref{ww}), (\ref{lia}). For $\mathfrak a \in \mathrm{FI}_F$, we have
\begin{equation} \label{ww} 
\log N\mathfrak a =\sum_{\iota \colon F \hookrightarrow \mathbb R} \log_\iota \mathfrak a.
\end{equation}
For a principal ideal $(\alpha)$ with $\alpha \in F^\times$, 
there exist $\epsilon \in \mathcal O^\times_+$, $m \in \mathbb N$ satisfying 
\begin{equation} \label{lia}
\log_\iota (\alpha)= \log |\iota(\alpha)| +\frac{1}{m}\log \iota (\epsilon)
\end{equation}
for all real embeddings $\iota$ of $F$.
Then we define the $W$-invariant by
\begin{equation*}
\begin{split}
W(c,\iota;D,{\mathfrak a}_c)&:=-\log_\iota {\mathfrak a}_c \mathfrak f \cdot \left[\sum_{z \in ({\mathfrak a}_c \mathfrak f)^{-1}\cap D,\, 
(z){\mathfrak a}_c \mathfrak f \in c }\iota(z)^{-s} \right]_{s=0} \\
&=-\log_\iota {\mathfrak a}_c \mathfrak f \cdot \sum_{j \in J}\sum_{\mathbf x \in R(c,{\mathfrak a}_c,\mathbf v_j)} 
\zeta(0,\iota(\mathbf v_j),\iota(\mathbf x\,{}^t\mathbf v_j)).
\end{split}
\end{equation*}
Finally, the $V$-invariant \cite[Chapter III, (3.7), (3.29)]{Yo} is defined by
\begin{equation*}
\begin{split}
V(c,\iota;D,{\mathfrak a}_c)&:=\frac{2}{n}\sum_{k=2}^n v_{1,k} -\frac{2}{n^2}\sum_{1\leq i < k \leq n} v_{i,k}, \\
v_{i,k}&:=\left[\frac{d}{ds}\sum_{z \in ({\mathfrak a}_c \mathfrak f)^{-1} \cap D,\,(z){\mathfrak a}_c \mathfrak f \in c} 
\left((\iota_i(z)\iota_k(z))^{-s} -\iota_i(z)^{-s}-\iota_k(z)^{-s} \right)\right]_{s=0},
\end{split}
\end{equation*}
where we put $\{\iota_1:=\iota,\iota_2,\dots,\iota_n\}$ to be all real embeddings of $F$.
We define
\begin{equation*}
\begin{split}
X(c,\iota;D,{\mathfrak a}_c):=G(c,\iota;D,{\mathfrak a}_c)+W(c,\iota;D,{\mathfrak a}_c)+V(c,\iota;D,{\mathfrak a}_c).
\end{split}
\end{equation*}

We consider fundamental domains of the following form, for the natural action $\epsilon (z \otimes \alpha):=(\epsilon z) \otimes \alpha$ 
of $\epsilon \in \mathcal O^\times_+$ on $z \otimes \alpha \in F\otimes \mathbb R_{n+}$.

\begin{dfn}
We say that a subset $D \subset F\otimes \mathbb R_{n+}$ is a Shintani domain if and only if we can write
\begin{equation*}
D =\coprod_{j\in J}C(\mathbf v_j), \quad F\otimes \mathbb R_{n+}=\coprod_{\epsilon \in \mathcal O^\times_+}\epsilon D
\end{equation*}
with a finite number of cones $C(\mathbf v_j)$ 
($j\in J$, $|J|<\infty$, $\mathbf v_j \in \mathcal O_+^{r(j)}$, $1\leq r(j) \leq n$).
\end{dfn}

Shintani showed that there exists a Shintani domain for any $F$ (\cite[Proposition 4]{Shin1}). 
If $D$ is a Shintani domain, and if $D,{\mathfrak a}_c$ are fixed, then $X(c,\iota;D,{\mathfrak a}_c)$ is also written as $X(c,\iota)$.
For later use, we introduce the following. When $D$ is a Shintani domain, by \cite[Chapter IV, Corollary 6.3-2]{Yo}, we have
\begin{equation} \label{=zeta}
\left[\sum_{z \in ({\mathfrak a}_c \mathfrak f)^{-1}\cap D,\, (z){\mathfrak a}_c \mathfrak f \in c }\iota(z)^{-s} \right]_{s=0}=\zeta(0,c) \in \mathbb Q.
\end{equation}
For a proof of the last part ($\zeta(0,c) \in \mathbb Q$), which seems to be well-known to experts, see e.g.\ \cite[Chapter II, Theorem 3.3]{Yo}.

The following Lemma, which is a part of Lemma \ref{replace2}, explains the reason why we modified the $W$-invariant in this paper:
If we replace $\iota (\mathcal O^\times_+)^\mathbb Q$ in the statement by $\iota (F_+)^\mathbb Q$, 
then it was proved by Yoshida \cite[Chapter III, \S 3.6, \S 3.7]{Yo} for the original $W$-invariant.

\begin{lmm}
$\exp(X(c,\iota)) \bmod \iota (\mathcal O^\times_+)^\mathbb Q$ does not depend on the choices of a Shintani domain $D$ 
and an integral ideal ${\mathfrak a}_c$. 
\end{lmm}

Now we state a modified version of Shintani's formula:
Shintani \cite{Shin2} expressed $\zeta'(0,c)$ in terms of $\log$ of Barnes' multiple gamma function, with certain correction terms.
Yoshida found a nice decomposition of the correction terms, which can be written as follows.

\begin{thm}[{\cite[Chapter III, (3.11)]{Yo}}] \label{candec}
Let $c, \mathfrak a_c$ be as above. Assume that $D$ is a Shintani domain. Then we have
\begin{equation} \label{mvsf}
\zeta'(0,c)=\sum_{\iota\colon F \hookrightarrow \mathbb R}X(c,\iota;D,\mathfrak a_c).
\end{equation}
Here $\iota$ runs over all real embeddings of $F$.
\end{thm}

\begin{proof}
In \cite[Chapter III, (3.11)]{Yo}, the equation (\ref{mvsf}) was proved for the original $W$-invariant (\ref{worig}).
Therefore we need to show that $\sum_{\iota} W(c,\iota;D,\mathfrak a_c)$ does not change when we modify the $W$-invariant.
This follows from (\ref{ww}) since 
$\left[\sum_{z \in ({\mathfrak a}_c \mathfrak f)^{-1}\cap D,\, (z){\mathfrak a}_c \mathfrak f \in c }\iota(z)^{-s} \right]_{s=0}$ does not depend on $\iota$ 
by (\ref{=zeta}).
\end{proof}

\section{Monomial relations between \mbox{\boldmath $\exp (X(c,\iota))$}'s.}

Let $F$ be a totally real filed of degree $n$, 
and $\mathfrak f=\mathfrak m\infty_1\dots\infty_n$ an integral divisor of $F$, as in the previous section.
We denote the maximal ray class field modulo $\mathfrak f$ by $H_\mathfrak f$.
Then the Artin map gives rise to a canonical isomorphism $C_\mathfrak f \cong \mathrm{Gal}(H_\mathfrak f/F)$.
Let $\iota_1,\dots,\iota_n$ be all real embeddings of $F$.
We denote the complex conjugation in $\mathrm{Gal}(H_\mathfrak f/F)$ at $\iota_i$ by $\rho_i$, 
and the corresponding element in $C_\mathfrak f$ by $c_i$.
That is, taking a lift $\tilde \iota_i \colon H_\mathfrak f \rightarrow \mathbb C$ of $\iota_i \colon F \rightarrow \mathbb R$, 
we put $\rho_i:=\tilde \iota_i^{-1} \circ \rho \circ \tilde \iota_i$
where $\rho$ is the complex conjugation on $\mathbb C$.
The following is the main result in this paper. 
When $n=2$, this Theorem is due to Yoshida \cite[Chapter III, Theorems 5.8, 5.12]{Yo}. 

\begin{thm} \label{main}
Assume that $n \geq 2$. Then there exist $\epsilon \in \mathcal O^\times_+$, $m \in \mathbb N$ satisfying 
\begin{equation*}
\exp (X(c,\iota_i)) \cdot \exp (X(c_j c,\iota_i)) = \iota_i (\epsilon)^\frac{1}{m}\ \text{ whenever }\ i \neq j \qquad (1\leq i,j \leq n).
\end{equation*}
\end{thm}

We prepare some Lemmas for the proof of Theorem \ref{main}.
The statement of Lemma \ref{cpxcnj} seems to be well-known to experts. For a proof, see \cite[Chapter III, the first paragraph of \S 5.1]{Yo}.

\begin{lmm} \label{cpxcnj}
For $1 \leq i \leq n$, take $\nu_i \in \mathcal O$ so that $\nu_i \equiv 1 \bmod \mathfrak m$, 
$\iota_i(\nu_i)<0$, $\iota_j(\nu_i)>0$ ($1\leq j \leq n$, $j\neq i$). 
Then we have $(\nu_i) \in c_i$.
\end{lmm}

We fix a numbering of real embeddings $\iota_1,\dots,\iota_n$ of $F$, and consider the following domain in $F\otimes \mathbb R=\mathbb R^n$:
\begin{equation*}
F\otimes \mathbb R_{(n-1)+}:=\mathbb R_+^{n-1}\times \mathbb R=\{(x_1,\dots,x_n) \in \mathbb R^n \mid x_1,\dots,x_{n-1}>0\}.
\end{equation*}
In the following, we consider (disjoint unions of) cones contained in $F\otimes \mathbb R_{(n-1)+}$ (not necessarily in $F\otimes \mathbb R_{n+}$).

\begin{lmm} \label{keylemma}
There exist a Shintani domain $D$, an element $\nu \in F$, 
subsets $X_t \subset F \otimes \mathbb R_{(n-1)+}$, and elements $\epsilon_t \in \mathcal O^\times_+$ ($t \in T$, $T$ is a finite set of indices) satisfying that 
\begin{enumerate}
\item Each $X_t$ has a cone decomposition (i.e., can be expressed as a disjoint union of a finite number of cones).
\item $\nu \in F \otimes \mathbb R_{(n-1)+}$ (i.e., $\iota_1(\nu),\dots,\iota_{n-1}(\nu)>0$, $\iota_n(\nu)<0$).
\item  We have the following equality of multisets:
\begin{equation*}
\left(D\coprod \nu D\right) \biguplus \left(\biguplus_{t \in T} \epsilon_t X_t \right)=\biguplus_{t \in T} X_t.
\end{equation*}
Here we denote by the symbol $\biguplus$ the multiset sum.
\end{enumerate}
\end{lmm}

We postpone the proof of Lemma \ref{keylemma} to \S \ref{proof} since it is rather technical. 
Here we only give an example in the case $[F:\mathbb Q]=2$.
We take the fundamental unit $\epsilon$ of a real quadratic field $F$ satisfying $|\iota_1(\epsilon)|>1$.
If $\epsilon$ is not totally positive, replace $\epsilon$ by $\epsilon^2$.
Then $D:=C(1) \coprod C(1,\epsilon)$ is a Shintani domain. 
Take any element $\nu \in \mathcal O$ satisfying $\iota_1(\nu)>0$, $\iota_2(\nu)<0$.
Then $X_1:=C(1,\nu)$, $\epsilon_1:=\epsilon$ ($T:=\{1\}$) satisfy the required conditions.

In order to prove Theorem \ref{main}, we generalize the definition of Yoshida's invariants 
to the case $D \subset F\otimes \mathbb R_{(n-1)+}$.
If $D \not\subset F\otimes \mathbb R_{n+}$, then (\ref{basic}) does not alway hold, so we need some modifications:
Consider a cone $C(\mathbf v)$ with $\mathbf v=(v_1,\dots,v_r) \in \mathcal O^r$, $C(\mathbf v) \subset F\otimes \mathbb R_{(n-1)+}$.
We assume that ${\mathfrak a}_c \mathfrak f$ and $c$ belong to the same narrow ideal class, as in the previous section.
We take $z \in ({\mathfrak a}_c \mathfrak f)^{-1} \cap C(\mathbf v)$ satisfying $z{\mathfrak a}_c \mathfrak f \in c$.
In particular, $(z{\mathfrak a}_c \mathfrak f,\mathfrak f)=1$.
Then for $1\leq i \leq r$ we have
\begin{equation*}
(z\pm v_i)(z)^{-1}=(1 \pm v_i/z) \in \overline{(1)} \text{ or }\overline{(\nu_n)},
\end{equation*}
where we denote by $\overline{(\alpha)}$ the ray class in $C_\mathfrak f$ of a principal ideal $(\alpha)$ ($\alpha=1,\nu_n$), 
$\nu_n$ is defined in Lemma \ref{cpxcnj}, and we additionally assume that $z- v_i \in C(\mathbf v)$ in the case $(z- v_i)(z)^{-1}$.
Therefore, instead of (\ref{basic}), we obtain 
\begin{multline*}
\left \{z \in ({\mathfrak a}_c \mathfrak f)^{-1} \cap C(\mathbf v),\, z{\mathfrak a}_c \mathfrak f \in c \text{ or } c_nc\right\} \\
=\coprod_{\mathbf x \in R(\{c,c_nc\},{\mathfrak a}_c,\mathbf v)}
\left\{(x_1+m_1)v_1+\dots + (x_r+m_r)v_r \mid 0\leq m_1,\dots,m_r \in \mathbb Z\right\}
\end{multline*}
by defining
\begin{equation*}
R(\{c,c_nc\},{\mathfrak a}_c,\mathbf v)
:=\left \{\mathbf x \in (\mathbb Q \cap (0,1])^r\mid z:=\sum_{i=1}^r x_i v_i \in ({\mathfrak a}_c \mathfrak f)^{-1},\, 
z{\mathfrak a}_c \mathfrak f \in c \text{ or } c_nc\right\}.
\end{equation*}
Hence we can write for $\iota \neq \iota_n$
\begin{equation*}
\sum_{z \in ({\mathfrak a}_c \mathfrak f)^{-1}\cap C(\mathbf v),\, (z){\mathfrak a}_c \mathfrak f \in c \text{ or } c_nc} \iota(z)^{-s}
=\sum_{\mathbf x \in R(\{c,c_nc\},{\mathfrak a}_c,\mathbf v)}\zeta(s,\iota(\mathbf v),\iota(\mathbf x\,{}^t\mathbf v)).
\end{equation*}
Therefore we can generalize the $G,W$-invariants as follows: 
Let $c,{\mathfrak a}_c$ be as above.
We assume that $D \subset F\otimes \mathbb R$ satisfies that
\begin{itemize}
\item $D$ has a cone decomposition of the form $D=\coprod_{j \in J}C(\mathbf v_j)$ ($|J|<\infty$, $\mathbf v_j \in \mathcal O^{r(j)}$, $1 \leq r(j) \leq n$).
\item $D\subset F\otimes \mathbb R_{(n-1)+}$. 
\end{itemize}
Then for $\iota \neq \iota_n$, we put 
\begin{equation*}
\begin{split}
G(\{c,c_nc\},\iota;D,{\mathfrak a}_c)
&:=\left[\frac{d}{ds}\sum_{z \in ({\mathfrak a}_c \mathfrak f)^{-1}\cap D,\, (z){\mathfrak a}_c \mathfrak f \in c \text{ or } c_nc}\iota(z)^{-s} \right]_{s=0} \\
&=\sum_{j \in J}\sum_{\mathbf x \in R(\{c,c_nc\},{\mathfrak a}_c,\mathbf v_j)}\log \Gamma(\iota(\mathbf x\,{}^t\mathbf v_j),\iota(\mathbf v_j)),
\end{split}
\end{equation*}
\begin{equation*}
\begin{split}
W(\{c,c_nc\},\iota;D,{\mathfrak a}_c)&:=-\log_\iota {\mathfrak a}_c \mathfrak f \cdot \left[\sum_{z \in ({\mathfrak a}_c \mathfrak f)^{-1}\cap D,\, 
(z){\mathfrak a}_c \mathfrak f \in c \text{ or } c_nc}\iota(z)^{-s} \right]_{s=0} \\
&=-\log_\iota {\mathfrak a}_c \mathfrak f \cdot \sum_{j \in J}\sum_{\mathbf x \in R(\{c,c_nc\},{\mathfrak a}_c,\mathbf v_j)} 
\zeta(0,\iota(\mathbf v_j),\iota(\mathbf x\,{}^t\mathbf v_j)).
\end{split}
\end{equation*}
Concerning the definition of the $V$-invariant in the case $D \subset F \otimes \mathbb R_{(n-1)+}$,  
the series
\begin{equation*}
\sum_{z \in ({\mathfrak a}_c \mathfrak f)^{-1} \cap D,\,(z){\mathfrak a}_c \mathfrak f \in c \text{ or } c_nc} 
\left((\iota_i(z)\iota_k(z))^{-s} -\iota_i(z)^{-s}-\iota_k(z)^{-s} \right)
\end{equation*}
may not converge when $i=n$ or $k=n$.
In order to avoid this problem, we introduce the following ``formal multiple zeta values'':
Let $\mathbf a:=(a_1,\dots,a_r), \mathbf x:=(x_1,\dots,x_r), m \in \mathbb N$ with $a_i>0$, $x_i\geq 0$, $\mathbf x \neq \mathbf 0$.
By abuse of notation, we put
\begin{equation} \label{assum}
\zeta(s,\mathbf a,\mathbf x):=
\zeta(s,\mathbf a,\mathbf x \,{}^t \mathbf a)=\sum_{0\leq m_1,m_2,\dots,m_r \in \mathbb Z}\left((x_1+m_1)a_1+\dots+(x_r+m_r)a_r\right)^{-s}.
\end{equation}
Then by \cite[Corollary to Proposition 1]{Shin1}, we have 
\begin{equation} \label{mzv}
\zeta(1-m,\mathbf a,\mathbf x)=(-1)^r(m-1)!\sum_{|\mathbf l|=m+r-1} \prod_{i=1}^r \frac{B_{l_i}(x_i)a_i^{l_i-1}}{l_i!}.
\end{equation}
Here we denote by $B_l(x)$ the $l$th Bernoulli polynomial,
and $\mathbf l=(l_1,\dots,l_r)$ runs over all $r$-tuples of non-negative integers satisfying $|\mathbf l|:=l_1+\dots + l_r=m+r-1$.
We define ``formal multiple zeta values'' as the same rational functions of $a_i,x_i$:
\begin{equation} \label{amzv}
\zeta_{\mathrm{fml}}(1-m,\mathbf a,\mathbf x):=(-1)^r(m-1)!\sum_{|\mathbf l|=m+r-1} \prod_{i=1}^r \frac{B_{l_i}(x_i)a_i^{l_i-1}}{l_i!}
\qquad (\mathbf a \in (\mathbb R^\times)^r,\ \mathbf x \in \mathbb R^r).
\end{equation}
For later use, we note the following relations.
We obtain 
\begin{equation} \label{abz1}
\zeta_{\mathrm{fml}}(1-m,\mathbf a,\mathbf x)=-\zeta_{\mathrm{fml}}(1-m,(-a_1,a_2,\dots,a_r),(1-x_1,x_2,\dots,x_r)),
\end{equation}
\begin{multline} \label{abz2}
\zeta_{\mathrm{fml}}(1-m,\mathbf a,(0,x_2,\dots,x_r)) =\zeta_{\mathrm{fml}}(1-m,\mathbf a,(1,x_2,\dots,x_r)) \\
+\zeta_{\mathrm{fml}}(1-m,(a_2,\dots,a_r),(x_2,\dots,x_r)),
\end{multline}
\begin{equation} \label{abz3}
\zeta_{\mathrm{fml}}(1-m,\mathbf a,\mathbf x)=\sum_{h=0}^{n-1} \zeta_{\mathrm{fml}}(1-m,(na_1,a_2,\dots,a_r),(x_1+h,x_2,\dots,x_r))
\end{equation}
by well-known formulas $B_l(1-x)=(-1)^lB_l(x)$, $B_l(0)=B_l(1)$ ($l\neq 1$), $B_1(0)=B_1(1)-1$, $B_l(nx)=n^{l-1}\sum_{h=0}^{n-1}B_l(x+h/n)$ ($n \in \mathbb N$).  
Furthermore we have
\begin{multline} \label{abz4}
\zeta_{\mathrm{fml}}(1-m,\mathbf a,\mathbf x)=\zeta_{\mathrm{fml}}(1-m,(a_1,a_1+a_2,a_3,\dots,a_r),(x_1-x_2+1,x_2,x_3,\dots,x_r)) \\
+\zeta_{\mathrm{fml}}(1-m,(a_1+a_2,a_2,a_3,\dots,a_r),(x_1,x_2-x_1,x_3,\dots,x_r)).
\end{multline}
We can show (\ref{abz4}) as follows:
When we can replace each $\zeta_{\mathrm{fml}}(\dotsm)$ by $\zeta(\dotsm)$ 
(that is, when $a_i>0$, $x_i\geq 0$, $1 \geq x_2-x_1 \geq 0$, and $\mathbf x,(x_1-x_2+1,x_2,x_3,\dots,x_r),(x_1,x_2-x_1,x_3,\dots,x_r) \neq \mathbf 0$),
this follows by decomposing the sum $\sum_{0\leq m_1,m_2,\dots,m_r \in \mathbb Z}$ in (\ref{assum}) as 
\begin{equation*}
\sum_{0\leq m_1,m_2,\dots,m_r \in \mathbb Z}
=\sum_{0\leq m_1,m_2,\dots,m_r \in \mathbb Z,\ m_1> m_2}
+\sum_{0\leq m_1,m_2,\dots,m_r \in \mathbb Z,\ m_1\leq m_2}.
\end{equation*}
Since $\zeta_{\mathrm{fml}}(1-m,\mathbf a,\mathbf x)$ is a rational function, the same equality holds for any $a_i,x_i$.
Similarly we can show that 
\begin{multline} \label{abz5}
\zeta_{\mathrm{fml}}(1-m,\mathbf a,\mathbf x)\\
=\zeta_{\mathrm{fml}}(1-m,(a_1,a_1+a_2,a_3,\dots,a_r),(x_1-x_2,x_2,x_3,\dots,x_r)) \\
+\zeta_{\mathrm{fml}}(1-m,(a_1+a_2,a_2,a_3,\dots,a_r),(x_1,x_2-x_1+1,x_3,\dots,x_r)).
\end{multline}
We can derive the following formula from (\ref{abz2}), (\ref{abz4}) by substituting $x:=x_1=x_2$.
\begin{multline} \label{abz6}
\zeta_{\mathrm{fml}}(1-m,\mathbf a,(x,x,x_3,\dots,x_r)) \\
=\zeta_{\mathrm{fml}}(1-m,(a_1,a_1+a_2,a_3,\dots,a_r),(1,x,x_3,\dots,x_r)) \\
\quad+\zeta_{\mathrm{fml}}(1-m,(a_1+a_2,a_2,a_3,\dots,a_r),(x,1,x_3,\dots,x_r)) \\
\quad+\zeta_{\mathrm{fml}}(1-m,(a_1+a_2,a_3,\dots,a_r),(x,x_3,\dots,x_r)).
\end{multline}
By combining (\ref{abz1}) with (\ref{abz2}), substituting $x_1=1$, 
we obtain
\begin{multline} \label{abz7}
\zeta_{\mathrm{fml}}(1-m,\mathbf a,(1,x_2,\dots,x_r))+\zeta_{\mathrm{fml}}(1-m,(-a_1,a_2,\dots,a_r),(1,x_2,\dots,x_r)) \\
+\zeta_{\mathrm{fml}}(1-m,(a_2,\dots,a_r),(x_2,\dots,x_r))=0.
\end{multline}

Next we introduce some notations and results in order to rearrange the expression of the $V$-invariant in \cite[Chapter III, (3.7)]{Yo}.
Let $A=(a_{ij})$ be an $n\times r$ matrix with $a_{ij}>0$, and $\mathbf x=(x_j)$ an $r$-dimensional vector with $x_j\geq 0$, $\mathbf x \neq \mathbf 0$.
For $\mathbf l=(l_1,\dots,l_r),j,k$ with $0\leq l_1,\dots,l_r$, $1\leq j,k \leq n$, $j\neq k$, we put
\begin{equation*}
C_{\mathbf l,j,k}(A):=\int_0^1 \left\{\prod_{m=1}^r (a_{jm} + a_{km} u)^{l_m-1} - \prod_{m=1}^r a_{jm}^{l_m-1}\right\} \frac{du}{u}.
\end{equation*}
If $a_{jq}a_{kp}-a_{jp}a_{kq}\neq 0$ for all $p,q$ with $p\neq q$, then we have \cite[Chapter I, Lemma 2.3]{Yo}
\begin{equation*}
C_{\mathbf l,j,k}(A)+C_{\mathbf l,k,j}(A)
=\sum_{p \text{ with }l_p=0}\prod_{q \neq p}\left(a_{jq}/a_{jp}-a_{kq}/a_{kp}\right)^{l_q-1}\log\tfrac{a_{jp}}{a_{kp}}.
\end{equation*}
Therefore we can write
\begin{equation} \label{vtozeta}
\begin{split}
&(-1)^r\sum_{|\mathbf l|=r}\left(C_{\mathbf l,j,k}(A)+C_{\mathbf l,k,j}(A)\right)\prod_{q=1}^r\frac{B_{l_q}(x_q)}{l_q!} \\
&=(-1)^r\sum_{|\mathbf l|=r}\left(\sum_{p \text{ with }l_p=0}\prod_{q \neq p}\left(a_{jq}/a_{jp}-a_{kq}/a_{kp}\right)^{l_q-1}\log\tfrac{a_{jp}}{a_{kp}}\right)
\prod_{q=1}^r\frac{B_{l_q}(x_q)}{l_q!} \\
&=-\sum_{p=1}^r\left((-1)^{r-1}\sum_{|\mathbf l|=r,\,l_p=0}\prod_{q\neq p}\frac{B_{l_q}(x_q)\left(a_{jq}/a_{jp}-a_{kq}/a_{kp}\right)^{l_q-1}}{l_q!}\right)
\log\tfrac{a_{jp}}{a_{kp}} \\
&=-\sum_{p=1}^r\zeta_{\mathrm{fml}}(-1,(a_{jq}/a_{jp}-a_{kq}/a_{kp})_{q\neq p},(x_q)_{q\neq p})\log\tfrac{a_{jp}}{a_{kp}}.
\end{split}
\end{equation}
Here for an $r$-dimensional vector $(x_q)=(x_1,\dots,x_r)$, we denote by $(x_q)_{q\neq p}$ the $(r-1)$-dimensional vector $(x_1,\dots,x_{p-1},x_{p+1},\dots,x_r)$.
We define the ``formal $V$-invariant'' as follows.

\begin{dfn} \label{defofvab}
Let $c, \mathfrak a_c, D=\coprod_{j \in J} C(\mathbf v_j) \subset F \otimes  \mathbb R_{(n-1)+}$ be as above.
Then for any real embedding $\iota$ of $F$, we define
\begin{equation} \label{dfnvab}
V_{\mathrm{fml}}(\{c,c_nc\},\iota;D,\mathfrak a_c):=\sum_{j \in J} \sum_{\mathbf x \in R(\{c,c_nc\},\mathfrak a_c,\mathbf v_j)} V(\mathbf v_j,\mathbf x,\iota).
\end{equation}
Here, for $\mathbf v=(v_1,v_2,\dots,v_r) \in \mathcal O^r$, $\mathbf x=(x_1,x_2,\dots,x_r) \in (\mathbb Q \cap (0,1])^r$, we put
\begin{multline*}
V(\mathbf v,\mathbf x,\iota):=
\frac{-1}{n}\sum_{k=2}^n \sum_{p=1}^r\zeta_{\mathrm{fml}}(-1,(\tfrac{\iota'_1(v_q)}{\iota'_1(v_p)}-\tfrac{\iota'_k(v_q)}{\iota'_k(v_p)})_{q\neq p},(x_q)_{q\neq p})
\log|\tfrac{\iota'_1(v_p)}{\iota'_k(v_p)}| \\
+\frac{1}{n^2}\sum_{1\leq i<k \leq n}\sum_{p=1}^r \zeta_{\mathrm{fml}}(-1,(\tfrac{\iota'_i(v_q)}{\iota'_i(v_p)}-\tfrac{\iota'_k(v_q)}{\iota'_k(v_p)})_{q\neq p},(x_q)_{q\neq p})
\log|\tfrac{\iota'_i(v_p)}{\iota'_k(v_p)}|
\end{multline*}
with $\{\iota'_1=\iota,\iota'_2,\dots,\iota'_n\}$ all real embeddings of $F$.
\end{dfn}

We note that, by (\ref{vtozeta}) and \cite[Chapter III, (3.7)]{Yo}, the equation (\ref{dfnvab}) holds 
if we replace $\{c,c_nc\},V_{\mathrm{fml}}(\cdots),\zeta_{\mathrm{fml}}(\cdots),D\subset F \otimes  \mathbb R_{(n-1)+}$ 
by $c,V(\cdots),\zeta(\cdots),D\subset F \otimes  \mathbb R_{n+}$ respectively.
In particular, if $ \{z \in ({\mathfrak a}_c \mathfrak f)^{-1} \cap C(\mathbf v),\, z{\mathfrak a}_c \mathfrak f \in c \text{ or } c_nc\}
= \{z \in ({\mathfrak a}_c \mathfrak f)^{-1} \cap C(\mathbf v),\, z{\mathfrak a}_c \mathfrak f \in c\}$ (e.g., if $c=c_nc$), then we have 
\begin{equation*}
V_{\mathrm{fml}}(\{c,c_nc\},\iota;D,\mathfrak a_c)=V(c,\iota;D,\mathfrak a_c).
\end{equation*}
The following Lemmas are modifications of \cite[Chapter III, \S 3.5, \S 3.6]{Yo}.

\begin{lmm} \label{welldef}
The definition of $V_{\mathrm{fml}}(\{c,c_nc\},\iota;D,\mathfrak a_c)$ 
does not depend on the choice of the cone decomposition of $D$.
\end{lmm}

\begin{proof}
Since the intersection of two cones is a disjoint union of a finite number of cones (or the empty set), it suffices to show that 
$V_{\mathrm{fml}}(\{c,c_nc\},\iota;D,\mathfrak a_c)$ does not change
when we replace the cone decomposition $D=\coprod_{j \in J} C(\mathbf v_j)$ by its refinement.
We have to consider the following cases for each $j \in J$. We write $\mathbf v:=\mathbf v_j$, $r:=r(j)$ for simplicity. 
\begin{enumerate}
\item[(I)] Change the order of the basis $\mathbf v=(v_{1},\dots,v_{r})$.
\item[(II)] Replace $v_{1}$ by $nv_{1}$ with $n \in \mathbb N$.
\item[(III)] Decompose $C(\mathbf v)$ into $C(\mathbf v^\sharp)\coprod C(\mathbf v^\flat) \coprod C(\mathbf v^\natural)$ with 
\begin{equation*}
\begin{split}
\mathbf v^\sharp&=(v^\sharp_1,v^\sharp_2,v^\sharp_3,\dots,v^\sharp_r):=(v_{1},v_{1}+v_{2},v_3,\dots,v_{r}), \\ 
\mathbf v^\flat&=(v^\flat_1,v^\flat_2,v^\flat_3,\dots,v^\flat_r):=(v_{1}+v_{2},v_{2},v_3,\dots,v_{r}), \\ 
\mathbf v^\natural&=(v^\natural_1,v^\natural_2,v^\natural_3,\dots,v^\natural_{r-1}):=(v_{1}+v_{2},v_{3},\dots,v_{r}). 
\end{split}
\end{equation*}
\end{enumerate}
We easily see that $V_{\mathrm{fml}}(\{c,c_nc\},\iota;D,\mathfrak a_c)$ does not change under (I), (II) by definition and (\ref{abz3}).
Let us consider the case (III).
We take $(x_1,\dots,x_r) \in R(\{c,c_nc\},\mathfrak a_c,\mathbf v)$ and put
\begin{equation*}
\begin{split}
(x^\sharp_1,x^\sharp_2,x^\sharp_3,\dots,x^\sharp_{r})&:=
\begin{cases}
(x_1-x_2+1,x_2,x_3,\dots,x_{r}) & \text{ if } x_1<x_2, \\
(x_1-x_2,x_2,x_3,\dots,x_{r}) & \text{ if } x_1>x_2, \\
(1,x,x_3,\dots,x_{r}) & \text{ if } x_1=x_2=:x, \\
\end{cases} \\
(x^\flat_1,x^\flat_2,x^\flat_3,\dots,x^\flat_{r})&:=
\begin{cases}
(x_1,x_2-x_1,x_3,\dots,x_{r}) & \text{ if } x_1<x_2, \\
(x_1,x_2-x_1+1,x_3,\dots,x_{r}) & \text{ if } x_1>x_2, \\
(x,1,x_3,\dots,x_{r}) & \text{ if } x_1=x_2=:x, \\
\end{cases} \\
(x^\natural_1,x^\natural_2,x^\natural_3,\dots,x^\natural_{r-1})&:=(x,x_3,x_4,\dots,x_r) \hspace{12pt} \text{ if } x_1=x_2=:x. 
\end{split}
\end{equation*}
Then we have 
\begin{multline*}
\left\{\sum_{i=1}^{r} x_iv_{i} +\sum_{i=1}^{r} m_iv_{i} \mid 0\leq m_1,m_2,\dots,m_{r} \in \mathbb Z\right\}\cap C(\mathbf v^*) \\
=
\begin{cases}
\left\{\displaystyle \sum_{i=1}^{r} x^*_iv^*_i +\sum_{i=1}^{r} m_iv^*_i \mid 0\leq m_1,m_2,\dots,m_{r} \in \mathbb Z\right\} & \text{ if } *=\sharp, \flat, \\
& \vspace{-7pt} \\
\left\{\displaystyle \sum_{i=1}^{r-1} x^*_iv^*_i +\sum_{i=1}^{r-1} m_iv^*_i \mid 0\leq m_1,m_2,\dots,m_{r-1} \in \mathbb Z\right\} & \text{ if } *=\natural,\,x_1=x_2, \\
& \vspace{-7pt} \\
\emptyset & \text{ if } *=\natural,\,x_1 \neq x_2.
\end{cases}
\end{multline*}
Therefore, when $x_1=x_2$, it suffices to show that for all $i,k$ with $1\leq i<k \leq n$ 

\begin{multline*}
\sum_{p=1}^r \zeta_{\mathrm{fml}}(-1,(\tfrac{\iota'_i(v_q)}{\iota'_i(v_p)}-\tfrac{\iota'_k(v_q)}{\iota'_k(v_p)})_{q\neq p},(x_q)_{q\neq p})
\log|\tfrac{\iota'_i(v_p)}{\iota'_k(v_p)}| \\
=\sum_{p=1}^r \zeta_{\mathrm{fml}}(-1,(\tfrac{\iota'_i(v^\sharp_q)}{\iota'_i(v^\sharp_p)}
-\tfrac{\iota'_k(v^\sharp_q)}{\iota'_k(v^\sharp_p)})_{q\neq p},(x^\sharp_q)_{q\neq p})
\log|\tfrac{\iota'_i(v^\sharp_p)}{\iota'_k(v^\sharp_p)}| \\
+\sum_{p=1}^r \zeta_{\mathrm{fml}}(-1,(\tfrac{\iota'_i(v^\flat_q)}{\iota'_i(v^\flat_p)}-\tfrac{\iota'_k(v^\flat_q)}{\iota'_k(v^\flat_p)})_{q\neq p},(x^\flat_q)_{q\neq p})
\log|\tfrac{\iota'_i(v^\flat_p)}{\iota'_k(v^\flat_p)}| \\
+\sum_{p=1}^{r-1} 
\zeta_{\mathrm{fml}}(-1,(\tfrac{\iota'_i(v^\natural_q)}{\iota'_i(v^\natural_p)}-\tfrac{\iota'_k(v^\natural_q)}{\iota'_k(v^\natural_p)})_{q\neq p},(x^\natural_q)_{q\neq p})
\log|\tfrac{\iota'_i(v^\natural_p)}{\iota'_k(v^\natural_p)}|.
\end{multline*}
When $x_1\neq x_2$, we drop the terms with the symbol ``$\natural$'' on the right-hand side.
By comparing the coefficients of $\log|\dotsm|$ on both sides, 
we can reduce the problem to the following relations among formal multiple zeta values:
\begin{multline*}
\zeta_{\mathrm{fml}}(-1,(\tfrac{\iota'_i(v_{q})}{\iota'_i(v_{p})}-\tfrac{\iota'_k(v_{q})}{\iota'_k(v_{p})})_{q\neq p},(x_q)_{q\neq p}) \\
=
\begin{cases}
\zeta_{\mathrm{fml}}(-1,(\tfrac{\iota'_i(v^\sharp_{q})}{\iota'_i(v^\sharp_{p})}-\tfrac{\iota'_k(v^\sharp_{q})}{\iota'_k(v^\sharp_{p})})_{q\neq p},(x^\sharp_q)_{q\neq p}) \\
+\zeta_{\mathrm{fml}}(-1,(\tfrac{\iota'_i(v^\flat_{q})}{\iota'_i(v^\flat_{p})}-\tfrac{\iota'_k(v^\flat_{q})}{\iota'_k(v^\flat_{p})})_{q\neq p},(x^\flat_q)_{q\neq p}) 
& (x_1\neq x_2,\,p\neq 1,2), \\
& \vspace{-7pt} \\
\zeta_{\mathrm{fml}}(-1,(\tfrac{\iota'_i(v^\sharp_{q})}{\iota'_i(v^\sharp_{p})}-\tfrac{\iota'_k(v^\sharp_{q})}{\iota'_k(v^\sharp_{p})})_{q\neq p},(x^\sharp_q)_{q\neq p}) \\
+\zeta_{\mathrm{fml}}(-1,(\tfrac{\iota'_i(v^\flat_{q})}{\iota'_i(v^\flat_{p})}-\tfrac{\iota'_k(v^\flat_{q})}{\iota'_k(v^\flat_{p})})_{q\neq p},(x^\flat_q)_{q\neq p}) \\
+\zeta_{\mathrm{fml}}(-1,(\tfrac{\iota'_i(v^\natural_{q})}{\iota'_i(v^\natural_{p-1})}-\tfrac{\iota'_k(v^\natural_{q})}{\iota'_k(v^\natural_{p-1})})_{q\neq p-1},
(x^\natural_q)_{q\neq p-1}) & (x_1= x_2,\,p\neq 1,2), 
\end{cases}
\end{multline*}
\begin{equation*}
\zeta_{\mathrm{fml}}(-1,(\tfrac{\iota'_i(v^\sharp_{q})}{\iota'_i(v^\sharp_{2})}-\tfrac{\iota'_k(v^\sharp_{q})}{\iota'_k(v^\sharp_{2})})_{q\neq 2},(x^\sharp_q)_{q\neq 2})
+\zeta_{\mathrm{fml}}(-1,(\tfrac{\iota'_i(v^\flat_{q})}{\iota'_i(v^\flat_{1})}-\tfrac{\iota'_k(v^\flat_{q})}{\iota'_k(v^\flat_{1})})_{q\neq 1},(x^\flat_q)_{q\neq 1})=0 
\qquad (x_1\neq x_2),
\end{equation*}
\begin{multline*}
\zeta_{\mathrm{fml}}(-1,(\tfrac{\iota'_i(v^\sharp_{q})}{\iota'_i(v^\sharp_{2})}-\tfrac{\iota'_k(v^\sharp_{q})}{\iota'_k(v^\sharp_{2})})_{q\neq 2},(x^\sharp_q)_{q\neq 2}) 
+\zeta_{\mathrm{fml}}(-1,(\tfrac{\iota'_i(v^\flat_{q})}{\iota'_i(v^\flat_{1})}-\tfrac{\iota'_k(v^\flat_{q})}{\iota'_k(v^\flat_{1})})_{q\neq 1},(x^\flat_q)_{q\neq 1}) \\
+\zeta_{\mathrm{fml}}(-1,(\tfrac{\iota'_i(v^\natural_{q})}{\iota'_i(v^\natural_{1})}
-\tfrac{\iota'_k(v^\natural_{q})}{\iota'_k(v^\natural_{1})})_{q\neq 1},(x^\natural_q)_{q\neq 1})=0 \qquad (x_1= x_2).
\end{multline*}
Indeed, the first relation follows from (\ref{abz4}), (\ref{abz5}), (\ref{abz6}), and
the remaining relations follow from (\ref{abz1}), (\ref{abz7}) respectively.
Then the assertion is clear.
\end{proof}

\begin{lmm} \label{vtoav}
Let $c,\mathfrak a_c$ be as above, $\{\iota'_1=\iota,\iota'_2,\dots,\iota'_n\}$ all real embeddings of $F$,  
and $C(\mathbf v)$ a cone with the basis $\mathbf v=(v_1,\dots,v_r) \in \mathcal O^r$.
We assume that $C(\mathbf v) \subset F\otimes \mathbb R_{(n-1)+}$.
\begin{enumerate}
\item Let $\alpha \in F_+ \cap \mathfrak a_c^{-1}$.
We assume that $\alpha^{-1}\mathbf v \in \mathcal O^r$ (multiplying $v_i$ by an integer, if necessary).
Then we have
\begin{multline*}
V_{\mathrm{fml}}(\{c,c_nc\},\iota;C(\alpha^{-1}\mathbf v),(\alpha)\mathfrak a_c)
- V_{\mathrm{fml}}(\{c,c_nc\},\iota; C(\mathbf v),\mathfrak a_c) \\
=\sum_{\mathbf x \in R(\{c,c_nc\},\mathfrak a_c,\mathbf v)} \left(\frac{1}{n}\log N(\alpha) - \log \iota (\alpha)\right) \\
\times \left(\iota(\zeta_{\mathrm{fml}}(0,\mathbf v,\mathbf x))
-\frac{1}{n} \sum_{k=1}^n \iota'_k(\zeta_{\mathrm{fml}}(0,\mathbf v,\mathbf x))\right).
\end{multline*}
\item Let $\nu_n$ be as in Lemma \ref{cpxcnj}. We assume that $\nu_n^{-1}\mathbf v \in \mathcal O^r$. Then we have 
\begin{multline*}
V_{\mathrm{fml}}(\{c,c_nc\},\iota;C(\nu_n^{-1}\mathbf v),(\nu_n)\mathfrak a_c)
- V_{\mathrm{fml}}(\{c,c_nc\},\iota; C(\mathbf v),\mathfrak a_c) \\
=\sum_{\mathbf x \in R(\{c,c_nc\},\mathfrak a_c,\mathbf v)} \left(\frac{1}{n}\log |N(\nu_n)| - \log \iota (\nu_n)\right) \\
\times \left(\iota(\zeta_{\mathrm{fml}}(0,\mathbf v,\mathbf x)) -\frac{1}{n} \sum_{k=1}^n \iota'_k(\zeta_{\mathrm{fml}}(0,\mathbf v,\mathbf x))\right).
\end{multline*}
Here we note that the roles of $c,c_nc$ in the symbol 
$V_{\mathrm{fml}}(\{c,c_nc\},\iota;C(\nu_n^{-1}\mathbf v),(\nu_n)\mathfrak a_c)$
are exchanged: $c_nc,(\nu_n)\mathfrak a_c$ belong to the same narrow ideal class, and we have $c_nc_nc=c$.
\end{enumerate}
\end{lmm}

\begin{proof}
First we prove (i). We have $R(\{c,c_nc\},\mathfrak a_c,\mathbf v) = R(\{c,c_nc\},(\alpha)\mathfrak a_c,\alpha^{-1}\mathbf v)$
by definition, so it suffices to show that 
\begin{multline} \label{alt}
V(\alpha^{-1} \mathbf v,\mathbf x,\iota'_1)- V(\mathbf v,\mathbf x,\iota'_1) \\
=\left(\frac{1}{n}\log N(\alpha) - \log \iota'_1 (\alpha) \right)\left(\iota'_1(\zeta_{\mathrm{fml}}(0,\mathbf v,\mathbf x))
-\frac{1}{n} \sum_{k=1}^n \iota'_k(\zeta_{\mathrm{fml}}(0,\mathbf v,\mathbf x))\right).
\end{multline}
Since $\zeta_{\mathrm{fml}}(-1,(\tfrac{\iota'_i(v_q)}{\iota'_i(v_p)}-\tfrac{\iota'_k(v_q)}{\iota'_k(v_p)})_{q\neq p},(x_q)_{q\neq p})$ does not change 
when we multiply $\mathbf v$ by $\alpha^{-1}$, we see that 
\begin{multline} \label{eq1}
V(\alpha^{-1} \mathbf v,\mathbf x,\iota'_1)-V(\mathbf v,\mathbf x,\iota'_1) \\
=\frac{1}{n}\sum_{k=2}^n 
\left(\sum_{p=1}^r\zeta_{\mathrm{fml}}(-1,(\tfrac{\iota'_1(v_q)}{\iota'_1(v_p)}-\tfrac{\iota'_k(v_q)}{\iota'_k(v_p)})_{q\neq p},(x_q)_{q\neq p})\right)
\log\tfrac{\iota'_1(\alpha)}{\iota'_k(\alpha)} \\
\quad -\frac{1}{n^2}\sum_{1\leq i<k \leq n}
\left(\sum_{p=1}^r\zeta_{\mathrm{fml}}(-1,(\tfrac{\iota'_i(v_q)}{\iota'_i(v_p)}-\tfrac{\iota'_k(v_q)}{\iota'_k(v_p)})_{q\neq p},(x_q)_{q\neq p})\right)
\log\tfrac{\iota'_i(\alpha)}{\iota'_k(\alpha)}.
\end{multline}
By definition, we have 
\begin{multline} \label{eq2}
\sum_{p=1}^r\zeta_{\mathrm{fml}}(-1,(\tfrac{\iota'_i(v_q)}{\iota'_i(v_p)}-\tfrac{\iota'_k(v_q)}{\iota'_k(v_p)})_{q\neq p},(x_q)_{q\neq p}) \\
=-(-1)^r\sum_{|\mathbf l|=r}\left(\sum_{p \text{ with }l_p=0}\prod_{q \neq p}\left(\frac{\iota'_i(v_q)}{\iota'_i(v_p)}-\frac{\iota'_k(v_q)}{\iota'_k(v_p)}\right)^{l_q-1}\right)
\prod_{q=1}^r\frac{B_{l_q}(x_q)}{l_q!}.
\end{multline}
We fix $\mathbf l=(l_1,\dots,l_r),i,k$ with $l_1+\dots +l_r =r$, $1\leq i < k \leq n$ and substitute the following values for the variables in Proposition \ref{prp} below:
Let $t:=|\{p \mid 1\leq p \leq r,\,l_p=0\}|$. Then there exist $1 \leq f_1,\dots,f_t,g_1,\dots,g_{r-t} \leq r$ satisfying 
$l_{f_1}=\dots=l_{f_t}=0$, $l_{g_1},\dots,l_{g_{r-t}}>0$, $\{f_1,\dots,f_t,g_1,\dots,g_{r-t}\}=\{1,2,\dots,r\}$.
We put $a_j:=\iota'_i(v_{f_j})^{-1}$, $a_j':=\iota'_k(v_{f_j})^{-1}$ ($1 \leq j \leq t$).
Concerning $b_1,\dots,b_t$, we define 
\begin{equation*}
\begin{split}
&b_1=b_2=\dots=b_{l_{g_1}-1}:=\iota'_i(v_{g_1}),\,b_{(l_{g_1}-1)+1}=b_{(l_{g_1}-1)+2}=\dots=b_{(l_{g_1}-1)+(l_{g_2}-1)}:=\iota'_i(v_{g_2}), \\
&\dots,b_{(l_{g_1}-1)+\dots+(l_{g_{r-t-1}}-1)+1}=b_{(l_{g_1}-1)+\dots+(l_{g_{r-t-1}}-1)+2}=\dots=b_t:=\iota'_i(v_{g_{r-t}}). 
\end{split}
\end{equation*}
For $b_1',\dots,b_t'$, we replace $\iota'_i$ by $\iota'_k$. 
Then Proposition \ref{prp} states that 
\begin{equation} \label{eq3}
\prod_{q=1}^r \iota'_i(v_q)^{l_q-1}-\prod_{q=1}^r \iota'_k(v_q)^{l_q-1}
=\sum_{p \text{ with }l_p=0}\prod_{q \neq p}\left(\frac{\iota'_i(v_q)}{\iota'_i(v_p)}-\frac{\iota'_k(v_q)}{\iota'_k(v_p)}\right)^{l_q-1}.
\end{equation}
By (\ref{eq1}), (\ref{eq2}), (\ref{eq3}) and (\ref{amzv}), we obtain
\begin{multline*}
V(\alpha^{-1} \mathbf v,\mathbf x,\iota'_1)- V(\mathbf v,\mathbf x,\iota'_1) 
=\frac{-1}{n}\sum_{k=2}^n 
\left(\iota'_1(\zeta_{\mathrm{fml}}(0,\mathbf v,\mathbf x)) -\iota'_k(\zeta_{\mathrm{fml}}(0,\mathbf v,\mathbf x))\right)
\log\tfrac{\iota'_1(\alpha)}{\iota'_k(\alpha)} \\
+\frac{1}{n^2}\sum_{1\leq i<k \leq n}
\left(\iota'_i(\zeta_{\mathrm{fml}}(0,\mathbf v,\mathbf x)) -\iota'_k(\zeta_{\mathrm{fml}}(0,\mathbf v,\mathbf x))\right)
\log\tfrac{\iota'_i(\alpha)}{\iota'_k(\alpha)}.
\end{multline*}
It is easy to see that the right-hand side of the above formula is equal to that of (\ref{alt}).
Then the assertion (i) is clear.
The same proof works for (ii) by using $R(\{c,c_nc\},\mathfrak a_c,\mathbf v) = R(\{c,c_nc\},(\nu_n)\mathfrak a_c,\nu_n^{-1}\mathbf v)$.
\end{proof}

\begin{prp} \label{prp}
Let $a_1,\dots,a_t, a_1',\dots,a_t', b_1,\dots,b_t, b_1',\dots,b_t'$ be variables.
Then we have 
\begin{equation} \label{poly}
\prod_{i=1}^t a_ib_i -\prod_{i=1}^t a_i'b_i'=\sum_{i=1}^t \left(\prod_{j=1}^t \left(a_ib_j-a_i'b_j'\right)\right)
\left(\prod_{1 \leq j \leq t,\ j\neq i} \left(\frac{a_i}{a_j}-\frac{a_i'}{a_j'}\right)^{-1}\right).
\end{equation}
\end{prp}

\begin{proof}
We put $d(i,i):=a_ib_i-a_i'b_i'$, $d(i,j):=(a_ib_j-a_i'b_j')/(a_i/a_j-a_i'/a_j')$ ($i \neq j$). 
Then the right-hand side of (\ref{poly}) can be written as $\sum_{i=1}^t \prod_{j=1}^t d(i,j)$.
We use mathematical induction. 
The case $t=1$ is trivial. Assume that (\ref{poly}) holds for $t$.
We easily see that for $i \leq t$
\begin{equation*}
a_{t+1}b_{t+1}+\frac{a_i'}{a_{t+1}'}\frac{d(t+1,t+1)}{a_i/a_{t+1}-a_i'/a_{t+1}'}=d(i,t+1).
\end{equation*}
Hence we can write
\begin{equation*}
\begin{split}
&\prod_{i=1}^{t+1} a_ib_i -\prod_{i=1}^{t+1} a_i'b_i' \\
&=\left(\prod_{i=1}^t a_ib_i -\prod_{i=1}^t a_i'b_i'\right) a_{t+1}b_{t+1} + d(t+1,t+1) \left(\prod_{i=1}^t a_i'b_i'\right) \\
&=\sum_{i=1}^t \left( \left(d(i,t+1)- \frac{a_i'}{a_{t+1}'}\frac{d(t+1,t+1)}{a_i/a_{t+1}-a_i'/a_{t+1}'} \right) \prod_{j=1}^t d(i,j) \right)
+ d(t+1,t+1) \prod_{i=1}^t a_i'b_i'.
\end{split}
\end{equation*}
Therefore it suffices to show that 
\begin{equation} \label{istst}
\sum_{i=1}^t \left(\frac{a_i'}{a_{t+1}'}\frac{-1}{a_i/a_{t+1}-a_i'/a_{t+1}'}\prod_{j=1}^t d(i,j)  \right)
=\prod_{j=1}^t d(t+1,j) - \prod_{i=1}^t a_i'b_i'.
\end{equation}
Substituting $a_i:=\frac{1}{a_{t+1}/a_i-a_{t+1}'/a_i'}$, $b_i:=a_{t+1}b_i-a_{t+1}'b_i'$ into (\ref{poly}), we obtain
\begin{equation*}
\prod_{j=1}^t d(t+1,j) - \prod_{i=1}^t a_i'b_i' = 
\sum_{i=1}^t \frac{\prod_{j=1}^t \left(\frac{a_{t+1}b_j-a_{t+1}'b_j'}{a_{t+1}/a_i-a_{t+1}'/a_i'}-a_i'b_j'\right)}
{\prod_{1 \leq j \leq t,\ j\neq i} \left(\frac{a_{t+1}/a_j-a_{t+1}'/a_j'}{a_{t+1}/a_i-a_{t+1}'/a_i'}-a_i'/a_j'\right)}. 
\end{equation*}
Then (\ref{istst}) follows since we have 
\begin{equation*}
\begin{split}
\frac{\frac{a_{t+1}b_j-a_{t+1}'b_j'}{a_{t+1}/a_i-a_{t+1}'/a_i'}-a_i'b_j'}{\frac{a_{t+1}/a_j-a_{t+1}'/a_j'}{a_{t+1}/a_i-a_{t+1}'/a_i'}-a_i'/a_j'}
&=d(i,j) \hspace{12pt} (i\neq j), \\
\frac{a_{t+1}b_i-a_{t+1}'b_i'}{a_{t+1}/a_i-a_{t+1}'/a_i'}-a_i'b_i'
&=\frac{a_i'}{a_{t+1}'}\frac{-d(i,i)}{a_i /a_{t+1}-a_i'/a_{t+1}'}.
\end{split}
\end{equation*}
Hence the assertion is clear.
\end{proof}

\begin{dfn} 
Let $c, \mathfrak a_c, D=\coprod_{j \in J} C(\mathbf v_j) \subset F \otimes  \mathbb R_{(n-1)+}$ be as above,
and $\iota$ a real embedding of $F$ other than $\iota_n$.
Then we define
\begin{equation*}
X_{\mathrm{fml}}(\{c,c_nc\},\iota;D,\mathfrak a_c)
:=G(\{c,c_nc\},\iota;D,\mathfrak a_c)+W(\{c,c_nc\},\iota;D,\mathfrak a_c)+V_{\mathrm{fml}}(\{c,c_nc\},\iota;D,\mathfrak a_c).
\end{equation*}
\end{dfn}

By the definitions of the $G,W$-invariants and Lemma \ref{welldef}, 
we see that the definition of $X_{\mathrm{fml}}(\{c,c_n\},\iota;D,\mathfrak a_c)$ does not depend on the choice of a cone decomposition of $D$.
We also see that if $D,D'$ satisfy $D \cap D' =\emptyset$ then we have
\begin{equation} \label{disj}
X_{\mathrm{fml}}(\{c,c_nc\},\iota;D \coprod D',\mathfrak a_c)=X_{\mathrm{fml}}(\{c,c_nc\},\iota;D,\mathfrak a_c)+X_{\mathrm{fml}}(\{c,c_nc\},\iota;D',\mathfrak a_c).
\end{equation}

The following Lemmas \ref{replace}, \ref{replace2} are modifications of \cite[Chapter III, (3.35), (3.46)]{Yo}.

\begin{lmm} \label{replace}
Let $c, \mathfrak a_c, \iota\neq \iota_n, D\subset F \otimes  \mathbb R_{(n-1)+}$ be as above. We put
\begin{equation*}
Z:=\left[\sum_{z \in (\mathfrak a_c \mathfrak f)^{-1} \cap D,\, (z)\mathfrak a_c \mathfrak f \in c \text{ or } c_nc} \iota(z)^{-s}\right]_{s=0}.
\end{equation*}
Then $Z \in \iota (F)$.
Moreover the following assertions hold.
\begin{enumerate}
\item Assume that $D$ is a Shintani domain. Then $Z \in \mathbb Q$.
More precisely, if $c=c_nc$, or if $\mathcal O^\times \cap F\otimes \mathbb R_{(n-1)+}=\emptyset$, then we have
\begin{equation*}
Z=\zeta(0,c).
\end{equation*}
Otherwise we have 
\begin{equation*}
Z=\zeta(0,c)+\zeta(0,c_nc).
\end{equation*}
\item For $\alpha \in F_+ \cap \mathfrak a_c^{-1}$, we have
\begin{multline*}
X_{\mathrm{fml}}(\{c,c_nc\},\iota;\alpha^{-1} D,(\alpha)\mathfrak a_c)-X_{\mathrm{fml}}(\{c,c_nc\},\iota;D,\mathfrak a_c) \\
= \left(\log \iota(\alpha)-\log_\iota (\alpha)\right) Z +\left(\frac{1}{n}\log N(\alpha)-\log\iota(\alpha) \right) 
\left(Z-\frac{1}{n} \mathrm{Tr}_{\iota(F)/\mathbb Q}\, Z \right).
\end{multline*}
Additionally, assume that $D$ is a Shintani domain. Then we have 
\begin{equation*}
X_{\mathrm{fml}}(\{c,c_nc\},\iota;\alpha^{-1} D,(\alpha)\mathfrak a_c)
-X_{\mathrm{fml}}(\{c,c_nc\},\iota;D,\mathfrak a_c) = \left(\log \iota(\alpha)-\log_\iota (\alpha)\right) Z.
\end{equation*}
\item For $\epsilon \in \mathcal O^\times_+$, we have
\begin{equation*}
X_{\mathrm{fml}}(\{c,c_nc\},\iota;\epsilon D,\mathfrak a_c)-X_{\mathrm{fml}}(\{c,c_n\},\iota;D,\mathfrak a_c)
= \frac{\mathrm{Tr}_{\iota(F)/\mathbb Q}\, Z}{n} \log \iota (\epsilon)
\end{equation*}
\end{enumerate}
In particular, if $D,D'$ are Shintani domains 
and $\mathfrak a_c,\mathfrak a_c'$ are integral ideals satisfying that $\mathfrak a_c\mathfrak f, \mathfrak a_c'\mathfrak f,c$ belong the same narrow ideal class,
then there exist $\epsilon \in \mathcal O_+^\times$, $m \in \mathbb N$ satisfying 
\begin{equation} \label{last}
X_{\mathrm{fml}}(\{c,c_nc\},\iota;D',\mathfrak a_c')-X_{\mathrm{fml}}(\{c,c_nc\},\iota;D,\mathfrak a_c)
= \frac{1}{m}\log \iota(\epsilon)
\end{equation}
for all real embeddings $\iota$ of $F$ with $\iota \neq \iota_n$.
\end{lmm}

\begin{proof}
The fact that $Z \in \iota(F)$ follows from (\ref{mzv}).
For (i), assume that $D$ is a Shintani domain.
First we claim that 
\begin{center} 
($\ast$) \quad if $\{z \in (\mathfrak a_c \mathfrak f)^{-1} \cap D \mid (z)\mathfrak a_c \mathfrak f \in c_nc\}\neq \emptyset$, 
then $\mathcal O^\times \cap F\otimes \mathbb R_{(n-1)+}\neq \emptyset$.
\end{center}
Indeed, take $z_1 \in (\mathfrak a_c \mathfrak f)^{-1} \cap D$ satisfying $(z_1)\mathfrak a_c \mathfrak f \in c_nc$. 
Then $z_1 \in F_+$ by $D \subset F\otimes \mathbb R_{n+}$, 
so both $c$ and $c_nc$ belong to the narrow ideal class of $\mathfrak a_c \mathfrak f$.
Hence the narrow ideal class of $(\nu_n)$ is equal to that of $(1)$, i.e., 
there exist $z_2 \in F_+$, $\epsilon_n \in \mathcal O^\times$ satisfying $1=z_2\nu_n \epsilon_n$.
Hence the claim ($\ast$) holds since $\epsilon_n\in \mathcal O^\times \cap F\otimes \mathbb R_{(n-1)+}$.
Therefore, if $\mathcal O^\times \cap F\otimes \mathbb R_{(n-1)+}=\emptyset$,
then $\{z \in (\mathfrak a_c \mathfrak f)^{-1} \cap D \mid (z)\mathfrak a_c \mathfrak f \in c \text{ or } c_nc\}
=\{z \in (\mathfrak a_c \mathfrak f)^{-1} \cap D \mid (z)\mathfrak a_c \mathfrak f \in c\}$, 
so we have $Z=\zeta(0,c) \in \mathbb Q$ by (\ref{=zeta}).
The same holds when $c=c_nc$. 
Next assume that $c\neq c_nc$ and that there exists an element $\epsilon_n \in \mathcal O^\times \cap F\otimes \mathbb R_{(n-1)+}$.
Let $\nu_n$ be as in Lemma \ref{cpxcnj}.
Then we have the following bijection:
\begin{equation*}
\begin{array}{ccc}
\{z \in ((\nu_n)\mathfrak a_c \mathfrak f)^{-1} \cap \nu_n^{-1}\epsilon_n D \mid (z)(\nu_n)\mathfrak a_c \mathfrak f \in c_nc\}
&\rightarrow& \{z \in (\mathfrak a_c \mathfrak f)^{-1} \cap D \mid (z)\mathfrak a_c \mathfrak f \in c_nc\}, \\
z &\mapsto& \nu_n\epsilon_n^{-1}z.
\end{array}
\end{equation*}
Namely, we can write
\begin{equation} \label{dcmp}
\begin{split}
&\{z \in (\mathfrak a_c \mathfrak f)^{-1} \cap D \mid (z)\mathfrak a_c \mathfrak f \in c \text{ or } c_nc\} \\
&=\{z \in (\mathfrak a_c \mathfrak f)^{-1} \cap D \mid (z)\mathfrak a_c \mathfrak f \in c\} \\
&\hspace{140pt} 
\coprod \nu_n\epsilon_n^{-1}\{z \in ((\nu_n)\mathfrak a_c \mathfrak f)^{-1} \cap \nu_n^{-1}\epsilon_n D \mid (z)(\nu_n)\mathfrak a_c \mathfrak f \in c_nc\}.
\end{split}
\end{equation}
Hence we obtain
\begin{equation*}
Z=\left[\sum_{z \in (\mathfrak a_c \mathfrak f)^{-1} \cap D,\, (z)\mathfrak a_c \mathfrak f \in c } \iota(z)^{-s}\right]_{s=0}
+\iota(\nu_n\epsilon_n^{-1})^0
\left[\sum_{z \in ((\nu_n)\mathfrak a_c \mathfrak f)^{-1} \cap \nu_n^{-1}\epsilon_n D,\ (z)(\nu_n)\mathfrak a_c \mathfrak f \in c_nc} \iota(z)^{-s}\right]_{s=0},
\end{equation*}
which is equal to $\zeta(0,c)+\zeta(0,c_nc) \in \mathbb Q$ by (\ref{=zeta}).
This complete the proof of (i).
For (ii), a similar statement was proved by Yoshida \cite[Chapter III, \S 3.7]{Yo}.
The same proof works, as follows.
Since
$\{z \in ((\alpha)\mathfrak a_c \mathfrak f)^{-1} \cap \alpha^{-1}D \mid (z)(\alpha)\mathfrak a_c \mathfrak f \in c \text{ or } c_nc\} \rightarrow 
\{z \in (\mathfrak a_c \mathfrak f)^{-1} \cap D \mid (z)\mathfrak a_c \mathfrak f \in c \text{ or } c_nc\}$, 
$z \mapsto \alpha z$ is a bijection, we have 
\begin{equation*}
\sum_{z \in ((\alpha)\mathfrak a_c \mathfrak f)^{-1} \cap \alpha^{-1}D,\, (z)(\alpha)\mathfrak a_c \mathfrak f \in c \text{ or } c_nc} \iota(z)^{-s}
=\iota(\alpha)^s\sum_{z \in (\mathfrak a_c \mathfrak f)^{-1} \cap D,\, (z)\mathfrak a_c \mathfrak f \in c \text{ or } c_nc} \iota(z)^{-s}.
\end{equation*}
Hence, by definition, we obtain
\begin{equation*}
\begin{split}
G(\{c,c_nc\},\iota;\alpha^{-1} D,(\alpha)\mathfrak a_c)&=G(\{c,c_nc\},\iota;D,\mathfrak a_c)+\log \iota(\alpha) \cdot Z, \\
W(\{c,c_nc\},\iota;\alpha^{-1} D,(\alpha)\mathfrak a_c)&=W(\{c,c_nc\},\iota;D,\mathfrak a_c) - \log_\iota (\alpha) \cdot Z.
\end{split}
\end{equation*}
Lemma \ref{vtoav}-(i) implies that 
\begin{multline*}
V_{\mathrm{fml}}(\{c,c_nc\},\iota;\alpha^{-1}D,(\alpha)\mathfrak a_c) \\
=V_{\mathrm{fml}}(\{c,c_nc\},\iota; D,\mathfrak a_c) 
+ \left( \frac{1}{n}\log N(\alpha) - \log \iota (\alpha) \right) \left(\iota(Z)-\frac{1}{n}\mathrm{Tr}_{\iota(F)/\mathbb Q}\, Z \right),
\end{multline*}
so the first assertion of (ii) follows.
If $D$ is a Shintani domain, then $Z \in \mathbb Q$ by (i), so we have $\iota(Z)-\frac{1}{n}\mathrm{Tr}_{\iota(F)/\mathbb Q}\, Z=0$.
This completes the proof of (ii). 
The assertion (iii) follows from (ii) by putting $\alpha:=\epsilon$, and by noting that $\log N(\epsilon)=\log_\iota (\epsilon)=0$.
Finally, we prove (\ref{last}).
By the assumption on $\mathfrak a_c,\mathfrak a_c'$, there exists $\alpha \in F_+ \cap \mathfrak a_c^{-1}$ satisfying $\mathfrak a_c'=(\alpha) \mathfrak a_c$.
Hence, by (ii) and (\ref{lia}), 
we can write $X_{\mathrm{fml}}(\{c,c_nc\},\iota;\alpha^{-1}D,\mathfrak a_c')-X_{\mathrm{fml}}(\{c,c_nc\},\iota;D,\mathfrak a_c)$ in the form 
$\frac{1}{m}\log \iota(\epsilon)$ with $\epsilon \in \mathcal O_+^\times$, $m \in \mathbb N$.
Since $D',\alpha^{-1}D$ are Shintani domains, by \cite[Chapter III, Lemma 3.13]{Yo}, we can write these in the form
\begin{equation*}
D'=\coprod_{j \in J} C(\mathbf v_j),\quad \alpha^{-1}D=\coprod_{j \in J} \epsilon_j C(\mathbf v_j)
\end{equation*}
with $\epsilon_j \in \mathcal O^\times_+$.
Therefore, by (iii), $X_{\mathrm{fml}}(\{c,c_nc\},\iota;D',\mathfrak a_c')-X_{\mathrm{fml}}(\{c,c_nc\},\iota;\alpha^{-1}D,\mathfrak a_c')$ also
can be written in the form $\frac{1}{m}\log \iota(\epsilon)$. Then the assertion (\ref{last}) is clear.
\end{proof}

\begin{lmm} \label{replace2}
Let $c \in C_\mathfrak f$.
If $D,D'$ are Shintani domains 
and $\mathfrak a_c,\mathfrak a_c'$ are integral ideals satisfying that $\mathfrak a_c\mathfrak f, \mathfrak a_c'\mathfrak f,c$ belong the same narrow ideal class,
then there exist $\epsilon \in \mathcal O_+^\times$, $m \in \mathbb N$ satisfying 
\begin{equation*}
X(c,\iota;D',\mathfrak a_c')-X(c,\iota;D,\mathfrak a_c)
= \frac{1}{m}\log \iota(\epsilon)
\end{equation*}
for all real embeddings $\iota$ of $F$.
\end{lmm}

\begin{proof}
We proceed similarly to the above proof.
Let $Z:=\left[\sum_{z \in (\mathfrak a_c \mathfrak f)^{-1} \cap D,\, (z)\mathfrak a_c \mathfrak f \in c} \iota(z)^{-s}\right]_{s=0}$ 
and $\alpha \in F_+ \cap \mathfrak a_c^{-1}$.
We easily see that
\begin{equation*}
\begin{split}
G(c,\iota;\alpha^{-1} D,(\alpha)\mathfrak a_c)&=G(c,\iota;D,\mathfrak a_c)+\log \iota(\alpha) \cdot Z, \\
W(c,\iota;\alpha^{-1} D,(\alpha)\mathfrak a_c)&=W(c,\iota;D,\mathfrak a_c) - \log_\iota (\alpha) \cdot Z.
\end{split}
\end{equation*}
In addition, by \cite[Chapter III, (3.44)]{Yo}, we have
\begin{equation*}
V(c,\iota;\alpha^{-1}D,(\alpha)\mathfrak a_c)=V(c,\iota; D,\mathfrak a_c)
+ \left( \frac{1}{n}\log N(\alpha) - \log \iota (\alpha) \right) \left(Z-\frac{1}{n}\mathrm{Tr}_{\iota(F)/\mathbb Q}\, Z \right).
\end{equation*}
Therefore we obtain
\begin{multline*}
X(c,\iota;\alpha^{-1} D,(\alpha)\mathfrak a_c) -X(c,\iota;D,\mathfrak a_c) \\
= \left(\log \iota(\alpha)-\log_\iota (\alpha)\right) Z
+ \left( \frac{1}{n}\log N(\alpha) - \log \iota (\alpha) \right) \left(Z-\frac{1}{n}\mathrm{Tr}_{\iota(F)/\mathbb Q}\, Z \right),
\end{multline*}
which is corresponding to Lemma \ref{replace}-(ii).
The remaining arguments are exactly the same as in the proof of Lemma \ref{replace}.
\end{proof}

\begin{lmm} \label{replace3}
Let $c, \mathfrak a_c, \iota\neq \iota_n$ be as above, $\nu_n$ as in Lemma \ref{cpxcnj}, and $D$ a Shintani domain.
Then the following assertions hold.
\begin{enumerate}
\item We have
\begin{equation*}
X_{\mathrm{fml}}(\{c,c_nc\},\iota;\nu_n^{-1} D,(\nu_n)\mathfrak a_c)-X_{\mathrm{fml}}(\{c,c_nc\},\iota;D,\mathfrak a_c) 
= \left(\log \iota(\nu_n)-\log_\iota (\nu_n)\right) \zeta(0,c).
\end{equation*}
Here we note that the roles of $c,c_nc$ in the symbol $X_{\mathrm{fml}}(\{c,c_nc\},\iota;\nu_n^{-1} D,(\nu_n)\mathfrak a_c)$ are exchanged.
\item If $c= c_nc$ or if $\mathcal O^\times \cap F\otimes \mathbb R_{(n-1)+}=\emptyset$, then we have 
\begin{equation*}
X_{\mathrm{fml}}(\{c,c_nc\},\iota; D,\mathfrak a_c)=X(c,\iota; D,\mathfrak a_c).
\end{equation*}
Otherwise, we have
\begin{multline*}
X_{\mathrm{fml}}(\{c,c_nc\},\iota; D,\mathfrak a_c) \\
=X(c,\iota;D,\mathfrak a_c)+X(c_nc,\iota;\nu_n^{-1}\epsilon_n D,(\nu_n)\mathfrak a_c) 
+\left(\log \iota(\nu_n\epsilon_n^{-1})-\log_\iota (\nu_n\epsilon_n^{-1}) \right)\zeta(0,c_nc).
\end{multline*}
\end{enumerate}
\end{lmm}

\begin{proof}
We again proceed similarly to the above proofs.
In particular, for (i), the same proof as in Lemma \ref{replace}-(ii) works, by using Lemma \ref{vtoav}-(ii) and 
a bijection $\{z \in ((\nu_n)\mathfrak a_c \mathfrak f)^{-1} \cap \nu_n^{-1}D \mid (z)(\nu_n)\mathfrak a_c \mathfrak f \in c \text{ or } c_nc\} \rightarrow 
\{z \in (\mathfrak a_c \mathfrak f)^{-1} \cap D \mid (z)\mathfrak a_c \mathfrak f \in c \text{ or } c_nc\}$, $z \mapsto \nu_n z$.
We prove (ii).
In the former case, we have 
$\{z \in (\mathfrak a_c \mathfrak f)^{-1} \cap D \mid (z)\mathfrak a_c \mathfrak f \in c \text{ or } c_nc\}=
\{z \in (\mathfrak a_c \mathfrak f)^{-1} \cap D \mid (z)\mathfrak a_c \mathfrak f \in c\}$ by ($\ast$) in the proof of Lemma \ref{replace}.
Then we have $G_{\mathrm{fml}}(\{c,c_nc\},\iota;D,\mathfrak a_c)=G(c,\iota;D,\mathfrak a_c)$, 
$W_{\mathrm{fml}}(\{c,c_nc\},\iota;D,\mathfrak a_c)=W(c,\iota;D,\mathfrak a_c)$ by definition.
Furthermore we obtain $V_{\mathrm{fml}}(\{c,c_nc\},\iota;D,\mathfrak a_c)=V(c,\iota;D,\mathfrak a_c)$, as we noted after Definition \ref{defofvab}.
Thus the assertion follows in this case.
Next, we assume that $c\neq c_nc$ and that there exists an element $\epsilon_n \in \mathcal O^\times \cap F\otimes \mathbb R_{(n-1)+}$.
We put $Z:=\left[\sum_{z \in ((\nu_n)\mathfrak a_c \mathfrak f)^{-1} \cap \nu_n^{-1}\epsilon_n D,\, 
(z)(\nu_n)\mathfrak a_c \mathfrak f \in c_nc} \iota(z)^{-s}\right]_{s=0}$, which is equal to $\zeta(0,c_nc) \in \mathbb Q$ by (\ref{=zeta}).
Then by (\ref{dcmp}), we have
\begin{equation*}
\begin{split}
G(\{c,c_nc\},\iota;\alpha^{-1} D,(\alpha)\mathfrak a_c)&=G(c,\iota;D,\mathfrak a_c)+G(c_nc,\iota;\nu_n^{-1}\epsilon_n D,(\nu_n)\mathfrak a_c)
+\log \iota(\nu_n\epsilon_n^{-1}) \cdot Z, \\
W(\{c,c_nc\},\iota;\alpha^{-1} D,(\alpha)\mathfrak a_c)&=W(c,\iota;D,\mathfrak a_c)+W(c_nc,\iota;\nu_n^{-1}\epsilon_n D,(\nu_n)\mathfrak a_c)
-\log_\iota (\nu_n\epsilon_n^{-1}) \cdot Z.
\end{split}
\end{equation*}
By further using (\ref{alt}) and $\iota(Z)-\frac{1}{n}\mathrm{Tr}_{\iota(F)/\mathbb Q}\, Z=0$, we obtain 
\begin{equation*}
V_{\mathrm{fml}}(\{c,c_nc\},\iota;\alpha^{-1}D,(\alpha)\mathfrak a_c) =V(c,\iota; D,\mathfrak a_c)+V(c_nc,\iota;\nu_n^{-1}\epsilon_n D,(\nu_n)\mathfrak a_c).
\end{equation*}
Summing up the above, we obtain the desired formula.
\end{proof}

\begin{proof}[Proof of Theorem \ref{main}]
By reordering the embeddings $\iota_1,\dots,\iota_n$ of $F$ if necessary, we may assume that $j=n$, $i< n$. 
We take $\nu_n \in \mathcal O$ as in Lemma \ref{cpxcnj}, $D, \nu, X_t, \epsilon_t$ ($t \in T$) as in Lemma \ref{keylemma}, 
and integral ideals $\mathfrak a_d$ satisfying that $\mathfrak a_d\mathfrak f$ and $d$ belong to the same narrow ideal class for $d=c,c_nc$.
For real numbers $a(\iota),b(\iota)$ associated with real embeddings $\iota$ of $F$, we write $a(\iota) \sim b(\iota)$ 
if and only if 
there exist $\epsilon \in \mathcal O_+^\times$, $m \in \mathbb N$ satisfying $a(\iota) - b(\iota)=\frac{1}{m} \log \iota(\epsilon)$ 
for all $\iota \neq \iota_n$.
This is an equivalence relation.
Then Theorem \ref{main} states that 
\begin{equation} \label{statement}
X(c,\iota; D,\mathfrak a_c)+X(c_nc,\iota; D,\mathfrak a_{c_nc}) \sim 0.
\end{equation}
On the other hand, by Lemma \ref{keylemma}-(iii), (\ref{disj}), Lemma \ref{replace}-(iii), we can write 
\begin{multline} \label{sim0}
X_{\mathrm{fml}}(\{c,c_nc\},\iota; D \coprod \nu D,\mathfrak a_c) \\
=\sum_{t \in T} X_{\mathrm{fml}}(\{c,c_nc\},\iota; X_t,\mathfrak a_c) - \sum_{t \in T} X_{\mathrm{fml}}(\{c,c_nc\},\iota; \epsilon_t X_t,\mathfrak a_c) \sim 0.
\end{multline}
In the following of the proof, we show that (\ref{statement}) follows from (\ref{sim0}).
We consider two cases: 
\begin{itemize}
\item Case 1. When $c= c_nc$ or when $\mathcal O^\times \cap F\otimes \mathbb R_{(n-1)+}=\emptyset$.
\item Case 2. Cases other than Case 1. 
\end{itemize}
In Case 1, we can write 
\begin{equation*}
\begin{split}
X_{\mathrm{fml}}(\{c,c_nc\},\iota; D \coprod \nu D,\mathfrak a_c)
&=X_{\mathrm{fml}}(\{c,c_nc\},\iota; D,\mathfrak a_c)+X_{\mathrm{fml}}(\{c,c_nc\},\iota; \nu D,\mathfrak a_c) \\
&=X(c,\iota; D,\mathfrak a_c)+X_{\mathrm{fml}}(\{c,c_nc\},\iota; \nu D,\mathfrak a_c) \\
&\sim X(c,\iota; D,\mathfrak a_c)+X_{\mathrm{fml}}(\{c,c_nc\},\iota; \nu_n^{-2}\nu D,(\nu_n^2)\mathfrak a_c) \\
&\sim X(c,\iota; D,\mathfrak a_c)+X_{\mathrm{fml}}(\{c,c_nc\},\iota; \nu_n^{-1}\nu D,(\nu_n)\mathfrak a_c) \\
&=X(c,\iota; D,\mathfrak a_c)+X(c_nc,\iota; \nu_n^{-1}\nu D,(\nu_n)\mathfrak a_c) \\
&\sim X(c,\iota; D,\mathfrak a_c)+X(c_nc,\iota; D,\mathfrak a_{c_nc}) 
\end{split}
\end{equation*}
by using (\ref{disj}), Lemma \ref{replace3}-(ii), (\ref{last}), 
Lemma \ref{replace3}-(i) and (\ref{lia}), 
Lemma \ref{replace3}-(ii), Lemma \ref{replace2}, respectively.
In Case 2, we take an element $\epsilon_n \in \mathcal O^\times \cap F\otimes \mathbb R_{(n-1)+}$. 
Then we obtain  
\begin{equation} \label{case2-1}
X_{\mathrm{fml}}(\{c,c_nc\},\iota; D \coprod \nu D,\mathfrak a_c) 
=X_{\mathrm{fml}}(\{c,c_nc\},\iota; D,\mathfrak a_c)+X_{\mathrm{fml}}(\{c,c_nc\},\iota; \nu D,\mathfrak a_c) \\
\end{equation}
by (\ref{disj}), 
\begin{equation} \label{case2-2}
X_{\mathrm{fml}}(\{c,c_nc\},\iota; D,\mathfrak a_c)
\sim X(c,\iota;D,\mathfrak a_c)+X(c_nc,\iota;\nu_n^{-1}\epsilon_n D,(\nu_n)\mathfrak a_c)
\end{equation}
by Lemma \ref{replace3}-(ii) and (\ref{lia}), and 
\begin{equation} \label{case2-3}
\begin{split}
X_{\mathrm{fml}}(\{c,c_nc\},\iota; \nu D,\mathfrak a_c) &\sim X_{\mathrm{fml}}(\{c,c_nc\},\iota; \nu_n^{-2}\nu D,(\nu_n^2)\mathfrak a_c) \\
&\sim X_{\mathrm{fml}}(\{c,c_nc\},\iota; \nu_n^{-1}\nu D,(\nu_n)\mathfrak a_c) \\
&\sim X(c_nc,\iota;\nu_n^{-1}\nu D,(\nu_n)\mathfrak a_c)+X(c,\iota;\nu_n^{-2}\epsilon_n \nu D,(\nu_n^2)\mathfrak a_c)
\end{split}
\end{equation}
by (\ref{last}), Lemma \ref{replace3}-(i) and (\ref{lia}), Lemma \ref{replace3}-(ii) and (\ref{lia}), respectively.
Hence, by (\ref{case2-1}), (\ref{case2-2}), (\ref{case2-3}) and Lemma \ref{replace2}, we obtain 
\begin{equation*}
X_{\mathrm{fml}}(\{c,c_nc\},\iota; D \coprod \nu D,\mathfrak a_c) \sim 2X(c,\iota; D,\mathfrak a_c)+2X(c_nc,\iota; D,\mathfrak a_{c_nc}).
\end{equation*}
Therefore we see that (\ref{sim0}) implies (\ref{statement}) in both cases.
This completes the proof.
\end{proof}

\begin{crl} \label{maincor}
Let $K$ be a CM-field which is an abelian extension of a totally real field $F$,
$\rho$ the unique complex conjugation in $G:=\mathrm{Gal}(K/F)$, 
$\mathfrak f_{K/F}$ the conductor of $K/F$, 
$\mathfrak f$ an integral divisor with $\mathfrak f_{K/F}|\mathfrak f$,
and $\mathrm{Art}_\mathfrak f$ the composite map $C_\mathfrak f \rightarrow C_{\mathfrak f_{K/F}} \stackrel{\mathrm{Art}}{\rightarrow} G$.
Then we have
\begin{equation*}
\prod_{c \in \mathrm{Art}_\mathfrak f^{-1}(\tau)}\exp (X(c,\iota)) \cdot \prod_{c \in \mathrm{Art}_\mathfrak f^{-1}(\tau \rho)} \exp (X(c,\iota)) 
\in 
\begin{cases}
\iota(\mathcal O^\times_+)^\mathbb Q & (F\neq\mathbb Q), \\
(K^\times \cap \mathbb R_+)^\mathbb Q & (F= \mathbb Q), \\
\end{cases}
\end{equation*}
for any real embedding $\iota$ of $F$ and $\tau \in G$.
Here $K^\times \cap \mathbb R_+$ in the case $F= \mathbb Q$ is well-defined since $K$ is normal over $\mathbb Q$.
\end{crl}

\begin{proof}
Let $c_j \in C_\mathfrak f$ be as in Theorem \ref{main}.
Since $K$ is a CM-field, $\mathrm{Art}_\mathfrak f(c_j)$ are equal to $\rho$ for all $j$.
Therefore we have $\{c \in \mathrm{Art}_\mathfrak f^{-1}(\tau \rho)\}=\{c_j c \mid c \in \mathrm{Art}_\mathfrak f^{-1}(\tau)\}$.
If $[F:\mathbb Q]\geq 2$, then there exists $j$ satisfying $\iota\neq \iota_j$, so the assertion follows from Theorem \ref{main}.
If $F=\mathbb Q$, then $X(c,\mathrm{id})=\zeta'(0,c)$, 
so the assertion is a part of Stark's conjecture for $(K \cap \mathbb R)/\mathbb Q$, which has been proved.
\end{proof}

\section{Proof of Lemma \ref{keylemma}.} \label{proof}

We prove Lemma \ref{keylemma} in this section.
We take an arbitrary Shintani domain $D=\coprod_{j\in J}C(\mathbf v_j)$ 
with $\mathbf v_j=(v_{j1},\dots,v_{jr(j)}) \in \mathcal O_+^{r(j)}$,
and put $J_n:=\{j \in J \mid r(j)=n \}$. 
We denote the set of all faces of $n$-dimensional cones $C(\mathbf v_j)$ ($j \in J_n$) by $F(D)$. Namely,
\begin{equation*}
F(D):=\{C(v_{ji_1},v_{ji_2},\dots,v_{ji_r}) \mid  j\in J_n,\,1\leq r\leq n-1,\,1\leq i_1<i_2<\dots<i_r\leq n\}.
\end{equation*} 
Then we see that 
\begin{equation} \label{top}
F \otimes \mathbb R_{n+} = \bigcup_{\epsilon \in \mathcal O^\times_+,\ j \in J_n} \epsilon \overline{C(\mathbf v_j)}
= \left(\coprod_{\epsilon \in \mathcal O^\times_+,\ j \in J_n} \epsilon C(\mathbf v_j)\right) \coprod  
\left(\bigcup_{\epsilon \in \mathcal O^\times_+,\ C \in F(D)} \epsilon C\right),
\end{equation}
where $\overline{C(\mathbf v_j)}$ denotes the topological closure of $C(\mathbf v_j)$ in $F \otimes \mathbb R_{n+}$.
We may assume the following (replacing the cone decomposition $D=\coprod_{j\in J}C(\mathbf v_j)$ by its refinement, if necessary).
\begin{itemize}
\item[($\star$)] If $C,C'\in F(D)$ and $\epsilon \in \mathcal O^\times_+$ satisfy $(\epsilon C) \cap C' \neq \emptyset$, then $\epsilon C= C'$. 
\item[($\star\star$)] If $C\in F(D)$, $j \in J-J_n$ and $\epsilon \in \mathcal O^\times_+$ satisfy $C(\mathbf v_j) \cap \epsilon C \neq \emptyset$, then 
$C(\mathbf v_j) \subset \epsilon C$.
\end{itemize}
In the following, we shall replace lower-dimensional cones $C(\mathbf v_j)$ ($j \in J-J_n$) so that 
$D$ can be expressed as the disjoint union of $n$-dimensional cones $C(\mathbf v_j)$ ($j \in J_n$) 
and their faces on ``the upper side'' with respect to the $x_n$-axis.
First we define which face is on ``the upper side'':
Consider the hyper plane $P:=\{(x_1,\dots,x_n) \in \mathbb R^n \mid x_n=0\}$.
We denote the line passing through two points $z,z' \in \mathbb R^n$ ($z\neq z'$) by $L(z,z'):=\{sz+(1-s)z' \mid s \in \mathbb R\}$.
Take $\nu \in F \cap F\otimes \mathbb R_{(n-1)+}$ such that $|\iota_1(\nu)-1|,|\iota_2(\nu)-1|,\dots,|\iota_{n-1}(\nu)-1|,|\iota_n(\nu)+1|$ are sufficiently small.
Then we may assume that for any $C=C(v_1,\dots,v_r) \in F(D)$, 
the angles between $L(v_i,\nu v_i)$ and $P$ ($i=1,\dots,r$) are closer to a right angle than that between $C$ and $P$.
(Note that $C \in F(D)$ and $P$ do not meet at a right angle since the basis of $C$ is in $F$.)
In particular, we see that for any $C \in F(D), z \in C$, the line $L(z,\nu z)$ is not contained in $C$. 
Namely, the assumption on $\nu$ implies that 
\begin{equation} \label{intprop}
L(z,\nu z) \cap C=\{z\} \qquad (C \in F(D),\ z \in C). 
\end{equation}
Actually, the following argument works whenever $\nu \in F \cap F\otimes \mathbb R_{(n-1)+}$ satisfies (\ref{intprop}).

\begin{prp} \label{atmostpt}
Let $\nu,D=\coprod_{j\in J}C(\mathbf v_j)$ be as above. Then $L(z,\nu z) \cap \epsilon C(\mathbf v_j)$ is a one-point set or the empty set 
for any $z \in F_+$, $j \in J-J_n$, $\epsilon \in \mathcal O^\times_+$.
\end{prp}

\begin{proof}
Assume that there exists an element $z_1 \in L(z,\nu z) \cap \epsilon C(\mathbf v_j)$.
By (\ref{top}) and ($\star\star$), there exist $\epsilon_1 \in \mathcal O^\times_+$ and $C \in F(D)$ satisfying 
$\epsilon C(\mathbf v_j) \subset \epsilon_1 C$.
Put $z_2:=\epsilon_1^{-1} z_1$. Then $z_2 \in C$, so $L(z_2,\nu z_2) \cap C$ is a one-point set by (\ref{intprop}).
We easily see that $L(z_2,\nu z_2)=\epsilon_1^{-1} L(z_1,\nu z_1)=\epsilon_1^{-1} L(z,\nu z)$,
$L(z,\nu z) \cap \epsilon C(\mathbf v_j) \subset L(z,\nu z) \cap \epsilon_1 C=\epsilon_1(\epsilon_1^{-1}L(z,\nu z) \cap  C)$.
Hence we have $L(z,\nu z) \cap \epsilon C(\mathbf v_j) \subset \epsilon_1(L(z_2,\nu z_2) \cap  C)$.
Then the assertion is clear.
\end{proof}

\begin{dfn}
Let $C(\mathbf v)$ be an $n$-dimensional cone with the basis $\mathbf v=(v_1,\dots,v_n) \in \mathcal O^n$.
We say that a face $C=C(v_{i_1},v_{i_2},\dots,v_{i_r})$ ($1\leq r\leq n-1$, $1\leq i_1<i_2<\dots<i_r\leq n$) of $C(\mathbf v)$ 
is on the upper side (resp.\ on the lower side) of $C(\mathbf v)$ 
if and only if for each $z \in C$, there exists $\delta>0$ satisfying 
$\{s\nu z+(1-s)z \mid 0<s<\delta \} \subset C(\mathbf v)$ (resp.\ $\{s\nu z+(1-s)z \mid -\delta<s<0 \} \subset C(\mathbf v)$).
\end{dfn}

We put $\{C(\mathbf v_t) \mid t \in T\}$ to be the set of all faces $\in F(D)$ which are on the upper side of some $C(\mathbf v_j)$ with $j \in J_n$.
We replace $D$ by 
\begin{equation*}
D_0:=\left(\coprod_{j \in J_n}C(\mathbf v_j)\right) \coprod \left(\coprod_{t \in T}C(\mathbf v_t)\right).
\end{equation*}

\begin{prp}
The above $D_0$ is again a Shintani domain.
\end{prp}

\begin{proof}
Let $z \in F \otimes \mathbb R_+-\coprod_{\epsilon \in \mathcal O^\times_+,\ j \in J_n}\epsilon C(\mathbf v_j)$.
It suffices to show that there exist unique $t \in T$, $z' \in C(\mathbf v_t)$, $\epsilon \in \mathcal O^\times_+$ satisfying $z=\epsilon z'$.
Take $\delta >0$ so that $I:=\{s\nu z+(1-s)z \mid 0\leq s \leq \delta\} \subset F\otimes \mathbb R_{n+}$.
Then we can write $I \subset \coprod_{k=1}^m \epsilon_k C(\mathbf v_{j_k})$ with $j_k \in J$, $\epsilon_k \in \mathcal O^\times_+$, $m \in \mathbb N$.
Replacing $\delta$ by a smaller positive number if necessary, we may assume that $m=2$. 
Then we can write 
$z \in \epsilon_1 C(\mathbf v_{j_1})$, $\{s\nu z+(1-s)z \mid 0< s < \delta\} \subset \epsilon_2 C(\mathbf v_{j_2})$
with $j_1 \in J-J_n$, $j_2 \in J_n$, $\epsilon_1,\epsilon_2 \in \mathcal O^\times_+$ by Proposition \ref{atmostpt}.
Therefore $z':=\epsilon_2^{-1} z$ is contained in a face on the upper side of $C(\mathbf v_{j_2})$, as desired.
In order to prove the uniqueness, 
assume that $t_k \in T$, $z_k \in C(\mathbf v_{t_k})$ ($k=1,2$), $\epsilon \in \mathcal O^\times_+$ satisfy $z_1=\epsilon z_2$.
Then there exist $\delta_k>0$ and $j_k \in J_n$ satisfying $\{s\nu z_k+(1-s)z_k \mid 0< s < \delta_k\} \subset C(\mathbf v_{j_k})$ ($k=1,2$).
In particular we have $C(\mathbf v_{j_1}) \cap \epsilon C(\mathbf v_{j_2}) \neq \emptyset$.
Since $\coprod_{j \in J_n,\ \epsilon \in \mathcal O^\times_+} \epsilon C(\mathbf v_j)$ is a disjoint union,
we have $C(\mathbf v_{j_1}) =C(\mathbf v_{j_2})$, $\epsilon=1$.
Hence we obtain $z_1=z_2$. This completes the proof.
\end{proof}

By definition, for $t \in T$, there exists a unique $j(t)\in J_n$ satisfying that $C(\mathbf v_t)$ is on the upper side of $C(\mathbf v_{j(t)})$.
On the other hand, for $t \in T$, there exist a unique $j(t)'\in J_n$ and a unique $\epsilon(t) \in \mathcal O^\times_+$ 
satisfying that $\epsilon(t)C(\mathbf v_t)$ is on the lower side of $C(\mathbf v_{j(t)'})$.
This follows by a similar argument to the above proof,
by using the line $\{s\nu z+(1-s)z \mid -\delta < s <0 \}$ and the assumption $(\star)$.
Put $X_t:=\{s\nu z + (1-s)z \mid z \in C(\mathbf v_t),\,0\leq s \leq 1\}$, $Y_t:=\epsilon(t) X_t$ for $t \in T$.
We write $\mathbf v_t=:(v_1,\dots,v_r)$, and take $n \in \mathbb N$ so that $n\nu \in \mathcal O$.
Then $X_t$ can be expressed as a finite disjoint union of cones whose bases are subsets of $\{v_1,\dots,v_r,n \nu v_1,\dots,n \nu v_r\}$.
Now we shall show that 
\begin{equation*}
\left(D_0 \coprod \nu D_0\right) \biguplus  \left(\biguplus_{t \in T} Y_t\right)=\biguplus_{t \in T} X_t.
\end{equation*}
For $j \in J_n$, we put $T_j:=\{t \in T \mid j(t)=j\}$, $T_j':=\{t \in T \mid j(t)'=j\}$, 
and $D_{0,j}:=C(\mathbf v_j) \coprod (\coprod_{t \in T_j} C(\mathbf v_t))$.
Then it suffices to show that 
\begin{equation} \label{decomp}
D_{0,j} \coprod \nu D_{0,j} \coprod \left(\coprod_{t \in T_j'} Y_t\right)=\coprod_{t \in T_j}X_t \qquad(j \in J_n).
\end{equation}
Here we note that
\begin{itemize}
\item $\{C(\mathbf v_t) \mid t \in T_j \}$ is the set of all faces on the upper side of $C(\mathbf v_j)$.
\item $\{\epsilon(t)C(\mathbf v_t) \mid t \in T_j' \}$ is the set of all faces on the lower side of $C(\mathbf v_j)$.
\item $\{\nu C(\mathbf v_t) \mid t \in T_j \}$ is the set of all faces on the lower side of $\nu C(\mathbf v_j)$.
\item $\{\epsilon(t)\nu C(\mathbf v_t) \mid t \in T_j' \}$ is the set of all faces on the upper side of $\nu C(\mathbf v_j)$.
\end{itemize}
We also note that $\nu C(\mathbf v_j)$ is almost equal to the ``mirror image'' of $C(\mathbf v_j)$ with respect to the hyperplane $P$.
Let $z_0 \in C(\mathbf v_{t_0})$ with $t_0 \in T_j$. 
Then the line segment $\{s\nu z_0 + (1-s)z_0 \mid 0\leq s \leq 1\}$ contains 
the point $z_0$ in the face $C(\mathbf v_{t_0})$ on the upper side of $C(\mathbf v_j)$, 
a point $z_1$ in a face on the lower side of $C(\mathbf v_j)$, 
a point $z_2$ in a face on the upper side of $\nu C(\mathbf v_j)$, 
and the point $\nu z_0$ in the face $\nu C(\mathbf v_j)$ on the lower side of $\nu C(\mathbf v_j)$.
We write $z_1=s_1\nu z + (1-s_1)z$, $z_2=s_2\nu z + (1-s_2)z$ with $0< s_1 < s_2 < 1$.
Then we can decompose $\{s\nu z_0 + (1-s)z_0 \mid 0\leq s \leq 1\} \subset X_{t_0}$ into 
\begin{multline*}
\{z_0\} \coprod \{s\nu z_0 + (1-s)z_0 \mid 0< s < s_1\} \coprod \{s\nu z_0 + (1-s)z_0 \mid s_1 \leq  s \leq s_2\} \\
\coprod \{s\nu z_0 + (1-s)z_0 \mid s_2< s < 1\} \coprod \{\nu z_0\},  
\end{multline*}
which are subsets of 
$\coprod_{t \in T_j} C(\mathbf v_t)$, $C(\mathbf v_j)$, $\coprod_{t \in T_j'} Y_t$, $\nu C(\mathbf v_j)$, $\nu \coprod_{t \in T_j} C(\mathbf v_t)$, respectively.
Hence we obtain the decomposition (\ref{decomp}).
This completes the proof of Lemma \ref{keylemma}.

\section{A relation between CM-periods and Stark units.}

In this section, we discuss a relation between monomial relations among $\exp(X(c,\iota))$'s, those among CM-periods, 
and Yoshida's conjecture on Shimura's period symbol $p_K$.
We note that some results (Propositions \ref{trans}, \ref{mainprp}, \ref{mainprp2}) are applications of Theorem \ref{main}.
First we recall some properties of $p_K$.

\begin{thm}[{\cite[Theorem 32.5]{Shim}}] \label{sps}
Let $K$ be a CM-field, $J_K$ the set of all embeddings of $K$ into $\mathbb C$, and 
$I_K:=\bigoplus_{\sigma \in J_K} \mathbb Z \sigma$ the free abelian group generated by elements in $J_K$. 
We denote the complex conjugation by $\rho$.
Then there exists a bilinear map $p_K\colon I_K \times I_K \rightarrow \mathbb C^\times/\overline{\mathbb Q}^\times$ satisfying the following properties.
\begin{enumerate}
\item Let $\Phi=\sum_{i=1}^{[K:\mathbb Q]/2} \sigma_i$ be a CM-type of $K$ (i.e., $\sigma_i \in J_K$, $\Phi+\Phi\rho=\sum_{\sigma \in J_K} \sigma$). 
Then $p_K(\sigma_i,\Phi)$ ($i=1,\dots,[K:\mathbb Q]/2$) are given as follows:
Let $A$ be an abelian variety with CM of type $(K,\Phi)$. 
Namely, for $i=1,\dots,[K:\mathbb Q]/2$, there exists a non-zero holomorphic differential $1$-form $\omega_i$ on $A$ 
where each $a \in K \cong \mathrm{End}(A)\otimes_\mathbb Z \mathbb Q$ acts as scalar multiplication by $\sigma_i(a)$.
We assume that $A,\omega_i$ are defined over $\overline{\mathbb Q}$.
Then we have
\begin{equation*}
\int_c \omega_i \bmod \overline{\mathbb Q}^\times =\pi p_K(\sigma_i,\Phi)
\end{equation*}
for every $c \in H_1(A(\mathbb C),\mathbb Z)$ with $\int_c \omega_i \neq 0$.
\item $p_K(\xi \rho ,\eta) p_K(\xi,\eta \rho)=p_K(\xi,\eta)^{-1}$ for $\xi,\eta \in I_K$.
\item Let $L$ also be a CM-field which is an extension of $K$. 
Then $p_K(\xi ,\mathrm{Res}(\zeta))=p_L(\mathrm{Inf}(\xi),\zeta)$, 
$p_K(\mathrm{Res}(\zeta),\xi)=p_L(\zeta,\mathrm{Inf}(\xi))$ for $\xi \in I_K$, $\zeta \in I_L$.
Here linear maps $\mathrm{Res}\colon I_L \rightarrow I_K$, $\mathrm{Inf}\colon I_K \rightarrow I_L$ are defined by 
$\mathrm{Res}(\sigma):=\sigma|_K$ ($\sigma  \in J_L$), 
$\mathrm{Inf}(\sigma):=\sum_{\tau \in J_L,\ \tau|_K=\sigma}\tau$ ($\sigma \in J_K$) respectively. 
\item Let $K'$ also be a CM-field with an isomorphism $\gamma\colon K' \cong K$. 
Then $p_{K'}(\gamma \xi,\gamma \eta)=p_K(\xi,\eta)$ for $\xi,\eta \in I_K$.
\end{enumerate}
\end{thm}

Let $K$ be a CM-field which is an abelian extension of a totally real field $F$.
We put $\mathfrak f_{K/F}$ to be the conductor and $G:=\mathrm{Gal}(K/F)$.
Hereinafter we take $\iota,D,\mathfrak a_c$ for Yoshida's invariant $X(c,\iota;D,\mathfrak a_c)$ in the following manner:
We regard a number field as a subfield of $\mathbb C$.
Then $\iota:=\mathrm{id} \in J_F$ has a meaning.
Fix a Shintani domain $D$ of $F$ and a complete system of representatives $\{\mathfrak a_\mu\}_{\mu}$ of the narrow ideal class group of $F$.
We assume that all of $\mathfrak a_\mu$ are integral ideals. 
We choose $\mathfrak a_c$ (satisfying that $\mathfrak a_c \mathfrak f$ and $c$ belong to the same narrow ideal class) 
from among $\{\mathfrak a_\mu\}_{\mu}$, and put $X(c):=X(c,\mathrm{id}; D,\mathfrak a_c)$.
Then Yoshida's absolute period symbol \cite[Chapter III, (3.10)]{Yo} is defined by
\begin{equation*}
g_K(\mathrm{id},\tau):=\pi^{-\mu(\tau)/2} \exp(\frac{1}{|G|} \sum_{\mathfrak f \mid \mathfrak f_{K/F}}\sum_{\chi \in (\hat G_-)_\mathfrak f}
\frac{\chi(\tau)}{L(0,\chi)} \sum_{c \in C_\mathfrak f} \chi(c) X(c))
\qquad (\tau \in G).
\end{equation*}  
Here we put $\mu(\tau):=1,-1,0$ if $\tau=\mathrm{id},\rho$, otherwise, respectively.
For an integral divisor $\mathfrak f$, we denote by $(\hat G_-)_{\mathfrak f}$ the set of all odd characters of $G$ whose conductors are equal to $\mathfrak f$.
Then we may regard each $\chi \in (\hat G_-)_{\mathfrak f}$ as a character of $C_\mathfrak f$.

\begin{cnj}[{\cite[Chapter III, Conjecture 3.9]{Yo}}] \label{YC}
Let $K,F,G$ be as above. Then we have
\begin{equation*}
g_K(\mathrm{id},\tau) \bmod \overline{\mathbb Q}^\times=p_K(\mathrm{id},\tau) \qquad (\tau \in G).
\end{equation*}
\end{cnj}

\begin{rmk}
Strictly speaking, Yoshida defined the symbol $g_K(\mathrm{id},\tau)$ by using 
the original $W$-invariant and formulated the above Conjecture, 
but this does not matter:
We see that $\frac{1}{|G|} \sum_{\chi \in (\hat G_-)_\mathfrak f} \frac{\chi(\tau)\chi(c)}{L(0,\chi)} \in \mathbb Q$ 
since $ (\hat G_-)_\mathfrak f$ is stable under the action of $\mathrm{Gal}( \overline{\mathbb Q}/\mathbb Q)$. 
Therefore $g_K(\mathrm{id},\tau) \bmod \overline{\mathbb Q}^\times$ does not change due to the modification of the $W$-invariant, 
since the $W$-invariants in \cite{Yo} and in this paper are of the form $a \log \alpha$ with $a \in \mathbb Q$, $\alpha \in \overline{\mathbb Q}$.
We also note that $g_K(\mathrm{id},\tau) \bmod \overline{\mathbb Q}^\times$
does not depend on the choices of a Shintani domain $D$ and integral ideals $\{\mathfrak a_\mu\}_\mu$ by \cite[Chapter III, \S 3.6, 3.7]{Yo}.
\end{rmk}

Under the above Conjecture, we can express any values of Shimura's period symbol, not only of the form $p_K(\mathrm{id},\tau)$, in terms of $\exp(X(c))$ 
(for details, see \cite[Chapter III, the paragraph following Conjecture 3.9]{Yo}).
Let us consider the opposite situation: a (conjectural) formula for $\exp(X(c))$ in terms of $p_K$.
For simplicity, we extend Shimura's period symbol $p_K$ to a bilinear map 
$\colon I_K\otimes_\mathbb Z \mathbb Q \times I_K\otimes_\mathbb Z \mathbb Q  \rightarrow \mathbb C^\times/\overline{\mathbb Q}^\times$,
which is also denoted by $p_K$. 

\begin{prp} \label{trans}
Conjecture \ref{YC} is equivalent to the following: Let $F,K,G,\mathfrak f_{K/F}$ be as above, 
$\mathfrak f_0$ an integral divisor of $F$ satisfying $\mathfrak f_{K/F} | \mathfrak f_0$. Then we have
\begin{equation} \label{equiv}
\prod_{c \in \mathrm{Art}_{\mathfrak f_0}^{-1}(\tau)}\exp(X(c)) \bmod \overline{\mathbb Q}^\times 
=\pi^{\zeta_{\mathfrak f_0}(0,\tau)}p_K(\tau,\sum_{\sigma \in G}\zeta_{\mathfrak f_0}(0,\sigma)\sigma).
\end{equation}
Here $\mathrm{Art}_{\mathfrak f_0}$ is the composite map $C_{\mathfrak f_0} \rightarrow C_{\mathfrak f_{K/F}} \stackrel{\mathrm{Art}}\rightarrow G$, and we put 
$\zeta_{\mathfrak f_0}(s,\sigma):=\zeta_S(s,\sigma)$ with $S:=\{$all places of $F$ dividing $\mathfrak f_0 \}$.
\end{prp}

\begin{proof}
First we note that $\zeta_{\mathfrak f_0}(0,\sigma) \in \mathbb Q$ by (\ref{=zeta}).
We prove that Conjecture \ref{YC} implies (\ref{equiv});
The converse follows by a similar but simpler argument.
The right-hand side of (\ref{equiv}) is equal to 
$\pi^{\zeta_{\mathfrak f_0}(0,\tau)}p_K(\mathrm{id},\sum_{\sigma \in G}\zeta_{\mathfrak f_0}(0,\sigma\tau)\sigma)$ by Theorem \ref{sps}-(iv).
Conjecture \ref{YC} states that this is equal to 
\begin{equation} \label{rhs}
\exp(\frac{1}{|G|} \sum_{\mathfrak f \mid \mathfrak f_{K/F}}\sum_{\chi \in (\hat G_-)_\mathfrak f}
\chi(\tau)^{-1} \prod_{\mathfrak p \mid \mathfrak f_0}  (1-\chi(\mathfrak p)) \sum_{c \in C_\mathfrak f} \chi(c) X(c)) \bmod \overline{\mathbb Q}^\times.
\end{equation}
Here we used the following relations:
\begin{equation*}
\begin{split}
&\sum_{\sigma \in G}\frac{-\mu(\sigma)}{2}\zeta_{\mathfrak f_0}(0,\sigma\tau)=
\frac{-1}{2}\zeta_{\mathfrak f_0}(0,\tau)+\frac{1}{2}\zeta_{\mathfrak f_0}(0,\tau\rho)=-\zeta_{\mathfrak f_0}(0,\tau), \\
&\sum_{\sigma \in G} \chi(\sigma)\zeta_{\mathfrak f_0}(0,\sigma\tau)=\chi(\tau)^{-1}L(0,\chi_{\mathfrak f_0})
=\chi(\tau)^{-1} \prod_{\mathfrak p \mid \mathfrak f_0} (1-\chi(\mathfrak p))L(0,\chi),
\end{split}
\end{equation*}
where the last equality in the first line follows from the fact that $L(0,\chi)=0$ for all even characters $\chi$, 
and the symbol $\chi_{\mathfrak f_0}$ in the second line denotes the composite map 
$C_{\mathfrak f_0} \rightarrow C_{\mathfrak f} \stackrel{\chi}\rightarrow \mathbb C^\times$.
Therefore it suffices to show that 
\begin{equation} \label{claim}
\text{$\prod_{\mathfrak p \mid \mathfrak f_0}  (1-\chi(\mathfrak p)) \sum_{c \in C_\mathfrak f} \chi(c) X(c)$ in (\ref{rhs}) 
can be replaced by $\sum_{c \in C_{\mathfrak f_0}} \chi_{\mathfrak f_0}(c) X(c)$}.
\end{equation}
Indeed by using (\ref{claim}), we see that (\ref{rhs}) is equal to 
\begin{multline*}
\exp(\frac{1}{|G|} \sum_{\chi \in \hat G_-}
\chi(\tau)^{-1} \sum_{c \in C_{\mathfrak f_0}} \chi_{\mathfrak f_0}(c) X(c)) \bmod \overline{\mathbb Q}^\times \\
=\exp(\frac{1}{2}\sum_{c \in \mathrm{Art}_{\mathfrak f_0}^{-1}(\tau)} X(c)
-\frac{1}{2}\sum_{c \in \mathrm{Art}_{\mathfrak f_0}^{-1}(\tau \rho)} X(c)) \bmod \overline{\mathbb Q}^\times.
\end{multline*}
Here we denote by $\hat G_-$ the set of all add characters of $G$, 
and we used the relation: $\sum_{\chi \in \hat G_-}\chi(\tau)=\frac{|G|}{2},-\frac{|G|}{2},0$ if $\tau=\mathrm{id},\rho$, otherwise, respectively.
Then the assertion follows from Corollary \ref{maincor}.

In the following of the proof, we show (\ref{claim}), which can be reduced to the following:
Let $\mathfrak f,\mathfrak f'$ be integral divisors satisfying 
$\mathfrak f | \mathfrak f_{K/F}$, $\mathfrak f | \mathfrak f' | \mathfrak f_0$. 
Then for any prime ideal $\mathfrak q | \frac{\mathfrak f_0}{\mathfrak f'}$, we have 
\begin{multline} \label{reduced}
\exp(\sum_{\chi \in (\hat G_-)_\mathfrak f}\chi(\tau)^{-1} \prod_{\mathfrak p \mid \mathfrak f_0} (1-\chi_{\mathfrak f'}(\mathfrak p)) 
\sum_{c \in C_{\mathfrak f'}} \chi_{\mathfrak f'}(c) X(c)) \\
\equiv 
\exp(\sum_{\chi \in (\hat G_-)_\mathfrak f}\chi(\tau)^{-1} \prod_{\mathfrak p \mid \mathfrak f_0} (1-\chi_{\mathfrak f'\mathfrak q}(\mathfrak p)) 
\sum_{c \in C_{\mathfrak f'\mathfrak q}} \chi_{\mathfrak f'\mathfrak q}(c) X(c))
\bmod \overline{\mathbb Q}^\times.
\end{multline}
In order to prove (\ref{reduced}), we recall a result from \cite{KY1}.
We note that $\{\mathfrak a_\mu\}_\mu$ is a complete system of representatives of the narrow ideal class group,
so is $\{\mathfrak a_\mu \mathfrak q\}_\mu$.
We write $X(c)$ as $X(c,\{\mathfrak a_\mu\}_\mu), X(c,\{\mathfrak a_\mu \mathfrak q\}_\mu)$ 
when we choose $\mathfrak a_c$ from among $\{\mathfrak a_\mu\}_{\mu}, \{\mathfrak a_c \mathfrak q\}_{\mu}$ respectively.
Then \cite[Lemma 5.3]{KY1} states that 
\begin{multline*}
\sum_{c \in C_{\mathfrak f'\mathfrak q}} \chi_{\mathfrak f'\mathfrak q}(c)X(c,\{\mathfrak a_\mu\}_\mu) \\
=\sum_{c \in C_{\mathfrak f'}} \chi_{\mathfrak f'}(c)X(c,\{\mathfrak a_\mu \mathfrak q\}_\mu)
-\chi_{\mathfrak f'}(\mathfrak q)\sum_{c \in C_{\mathfrak f'}} \chi_{\mathfrak f'}(c)X(c,\{\mathfrak a_\mu\}_\mu)
+\chi_{\mathfrak f'}(\mathfrak q)L(0,\chi_{\mathfrak f'})\log_{\mathrm{id}} \mathfrak q.
\end{multline*}
On the other hand, by Lemma \ref{replace2}, there exists an element $\alpha_{c,\{\mathfrak a_\mu\},\mathfrak q} \in (F_+)^\mathbb Q$
satisfying 
\begin{equation*}
X(c,\{\mathfrak a_\mu \mathfrak q\}_\mu)=X(c,\{\mathfrak a_\mu\}_\mu) + \log \alpha_{c,\{\mathfrak a_\mu\},\mathfrak q}.
\end{equation*}
Hence the ratio of both sides of (\ref{reduced}) can be written as the form
\begin{equation} \label{reduced2}
\exp(\sum_{i=1}^k \sum_{\chi \in (\hat G_-)_\mathfrak f} \chi(\mathfrak a_i) \log(\alpha_i)).
\end{equation}
Here $k$ is a natural number, $\mathfrak a_i$ are integral ideals, and $\alpha_i \in (F_+)^\mathbb Q$.
More precisely, each $\log \alpha_i$ is equal to one of $\pm \log \alpha_{c,\{\mathfrak a_\mu\},\mathfrak q},\pm \zeta_{\mathfrak f'}(0,c)\log \mathfrak q$, 
and each $\mathfrak a_i$ is a certain product of prime ideals dividing $\mathfrak f_0$, 
an integral ideal whose image under the Artin map is $\tau$, and representatives of classes in $C_{\mathfrak f'}$.
In particular, $\mathfrak a_i$'s and $\alpha_i$'s do not depend on $\chi \in (\hat G_-)_\mathfrak f$.
In addition, we have $\sum_{\chi \in (\hat G_-)_\mathfrak f} \chi(\mathfrak a_i) \in \mathbb Q$ since 
$(\hat G_-)_\mathfrak f$ is stable under the action of $\mathrm{Gal}(\overline{\mathbb Q}/\mathbb Q)$.
Hence we see that (\ref{reduced2}) is an algebraic number, so (\ref{reduced}) follows.
This completes the proof.
\end{proof}

The expression (\ref{equiv}) and Theorem \ref{sps}-(iii) suggest that it is natural to generalize Conjecture \ref{YC} to the following form.

\begin{cnj} \label{maincnj}
Let $K$ be a finite abelian extension of a totally real field $F$.
We assume that there exists a CM-subfield of $K$ and put $K_{\mathrm{CM}}$ to be the maximal CM-subfield of $K$.
Let $G:=\mathrm{Gal}(K/F), \mathfrak f_{K/F}| \mathfrak f_0, \mathrm{Art}_{\mathfrak f_0}$ be as in Proposition \ref{trans}. Then we have
\begin{equation*}
\prod_{c \in \mathrm{Art}_{\mathfrak f_0}^{-1}(\tau)}\exp(X(c)) \bmod \overline{\mathbb Q}^\times 
=\pi^{\zeta_{\mathfrak f_0}(0,\tau)}
p_{K_{\mathrm{CM}}}(\tau|_{{K_\mathrm{CM}}},\mathrm{Inf}^{-1}(\sum_{\sigma \in G}\zeta_{\mathfrak f_0}(0,\sigma)\sigma))
\qquad (\tau \in G).
\end{equation*}
\end{cnj}

For the well-definedness of the right-hand side, we need to show that $\sum_{\sigma \in G}\zeta_{\mathfrak f_0}(0,\sigma)\sigma$ is 
in the image of $\mathrm{Inf} \colon I_{K_{\mathrm{CM}}}\otimes_\mathbb Z \mathbb Q \rightarrow I_K\otimes_\mathbb Z \mathbb Q$.
Although this fact may be well-known to experts, we give a brief proof here:
An algebraic number field $L \subset \mathbb C$ is a CM-field if and only if 
the complex conjugation $\rho$ on $\mathbb C$ satisfies 
that $\rho$ induces a non-trivial automorphism of $L$, and that $\rho \tau=\tau \rho$ for every $\tau \in J_L$ (\cite[Proposition 5.11]{Shim}). 
Therefore 
\begin{equation} \label{kcm}
\text{$K_{\mathrm{CM}}$ is the fixed subfield of $K$ under $\langle \rho_\iota \rho_{\iota'} \mid \iota,\iota' \in J_F \rangle $}, 
\end{equation}
where $\rho_\iota \in G$ is the complex conjugation at $\iota$.
Hence it suffices to show that $\zeta_{\mathfrak f_0}(0,\sigma )=\zeta_{\mathfrak f_0}(0,\sigma \rho_\iota \rho_{\iota'})$ for all $\sigma \in G$, 
$\iota,\iota' \in J_F $, which follows from the fact that $L(0,\chi)=0$ if $\chi$ is not totally odd.
More precisely, we can write  
\begin{equation} \label{inf}
\begin{split}
\zeta_{\mathfrak f_0}(0,\sigma|_{K_{\mathrm{CM}}})&=[K:K_{\mathrm{CM}}]\zeta_{\mathfrak f_0}(0,\sigma) \qquad (\sigma \in G),\\
\sum_{\sigma \in \mathrm{Gal}(K_{\mathrm{CM}}/F)}\zeta_{\mathfrak f_0}(0,\sigma)\sigma
&=[K:K_{\mathrm{CM}}]\,\mathrm{Inf}^{-1}(\sum_{\sigma \in G}\zeta_{\mathfrak f_0}(0,\sigma)\sigma).
\end{split}
\end{equation}

Conjecture \ref{maincnj} also implies a part of Stark's conjecture in the following sense.

\begin{prp} \label{mainprp}
Let $F,K,S$ be as in Conjecture \ref{wsc}.
Assume that there exists an integral divisor $\mathfrak f_0$ satisfying that 
\begin{itemize}
\item The conductor $\mathfrak f_{K/F}$ of $K/F$ divides $\mathfrak f_0$.
\item Any place dividing $\mathfrak f_0$ is in $S$.
\item The maximal ray class field $H_{\mathfrak f_0}$ modulo $\mathfrak f_0$ contains a CM-subfield.
\end{itemize}
Then Conjecture \ref{maincnj} implies that we have 
\begin{equation} \label{wwsc}
\exp (\zeta_S'(0,\tau)) \in \overline{\mathbb Q}^\times \qquad(\tau \in \mathrm{Gal}(K/F)).
\end{equation}
\end{prp}

\begin{proof}
Let $\mathfrak f_0$ satisfy the above assumptions.
By Theorem \ref{main}, for $\iota \neq \mathrm{id}$ we have
\begin{equation*}
\prod_{c \in \mathrm{Art}_{\mathfrak f_0}^{-1}(\tau)}\exp(X(c,\iota)) 
\cdot \prod_{c \in \mathrm{Art}_{\mathfrak f_0}^{-1}(\tau \rho_{\mathrm{id}})}\exp(X(c,\iota))
\in \overline{\mathbb Q}^\times \qquad (\tau \in \mathrm{Gal}(H_{\mathfrak f_0}/F)), 
\end{equation*}
where $\rho_{\mathrm{id}}$ is the complex conjugation on $H_{\mathfrak f_0}$ at $\mathrm{id} \in J_F$.
Conjecture \ref{maincnj} implies that the same holds for $\iota=\mathrm{id}$, by Theorem \ref{sps}-(ii).
Hence we obtain $\exp (\zeta_{S'}'(0,\tau))\exp (\zeta_{S'}'(0,\tau \rho_{\mathrm{id}})) \in \overline{\mathbb Q}^\times$ 
for $S':=\{$all places of $F$ dividing $\mathfrak f_0 \}$, $\tau \in \mathrm{Gal}(H_{\mathfrak f_0}/F)$.
Then the assertion follows, since $S'\subset S$ and $K$ is fixed under $\rho_{\mathrm{id}}$.
\end{proof}

\begin{exm} \label{exm}
We introduce an example which is discussed in \cite[Chapter III, Example 6.3]{Yo}.
Let $F=\mathbb Q(\sqrt 5)$, $K=F(\sqrt{\epsilon})$ with $\epsilon=\frac{1+\sqrt 5}{2}$.
We denote by $\infty_1,\infty_2$ the infinite places of $F$ corresponding to the real embeddings 
$\iota_1=\mathrm{id},\iota_2\colon \sqrt{5} \mapsto -\sqrt{5}$, respectively.
Then the conductor of $K/F$ is $(4) \infty_2$.
We put $S:=\{(2),\infty_1,\infty_2\}$, $\mathfrak f_0:=(4)\infty_1\infty_2$. 
Then the maximal ray class field $H_{\mathfrak f_0}$ is $F(\sqrt{\epsilon},\sqrt{-1})$, 
and its maximal CM-subfield is $F(\sqrt{-1})$.
We see that $C_{\mathfrak f_0}=\{c_1,c_2,c_3,c_4\}$ where 
$c_1,c_2,c_3,c_4$ are the classes of $(1),(3),(4+\sqrt 5),(6+\sqrt 5)$ respectively.
The element in $\mathrm{Gal}(H_{\mathfrak f_0}/F)$ corresponding to $c_i$ under the Artin map $C_{\mathfrak f_0} \cong \mathrm{Gal}(H_{\mathfrak f_0}/F)$
is denoted by $\tau_i$.
We put $G:=\frac{\Gamma(\frac{1}{4})}{\Gamma(\frac{3}{4})}\prod_{a=1}^{19} \Gamma(\frac{a}{20})^{\psi(a)/4}$ with $\psi$ the Dirichlet character 
corresponding to the quadratic extension $\mathbb Q(\sqrt{-5})/\mathbb Q$.
The following was shown in \cite[Chapter III, Example 6.3]{Yo}:
\begin{equation} \label{explicit}
\exp(X(c_1)) \equiv  \exp(X(c_2)) \equiv  G^{\frac{1}{4}}, \quad 
\exp(X(c_3)) \equiv \exp(X(c_4)) \equiv G^{-\frac{1}{4}} \bmod \overline{\mathbb Q}^\times, 
\end{equation}
\begin{equation*}
\begin{split}
&\zeta(0,c_1)=\zeta(0,c_2)=\frac{1}{4}, \quad \zeta(0,c_3)=\zeta(0,c_4)=-\frac{1}{4}, \\
&G \bmod \overline{\mathbb Q}^\times =\pi p_{F(\sqrt{-1})}(\mathrm{id},\mathrm{id})^2.
\end{split}
\end{equation*}
Therefore we see that Conjecture \ref{maincnj} holds true for $H_{\mathfrak f_0}/F$:
For example, let $\tau:=\tau_1=\mathrm{id} \in \mathrm{Gal}(H_{\mathfrak f_0}/F)$.
Then Conjecture \ref{maincnj} states that 
\begin{equation} \label{conj4}
\exp(X(c_1)) \bmod \overline{\mathbb Q}^\times 
=\pi^{\frac{1}{4}}
p_{F(\sqrt{-1})}(\mathrm{id},\mathrm{Inf}^{-1}(\tfrac{1}{4}\tau_1+\tfrac{1}{4}\tau_2-\tfrac{1}{4}\tau_3-\tfrac{1}{4}\tau_4)).
\end{equation}
Since $F(\sqrt{-1})$ is the fixed subfield of $H_{\mathfrak f_0}$ under $\{\tau_1,\tau_2\}$, 
we see that the right-hand side of (\ref{conj4}) is equal to $\pi^{\frac{1}{4}} p_{F(\sqrt{-1})}(\mathrm{id},\tfrac{1}{4} \mathrm{id}-\tfrac{1}{4}\rho)$,
which is again equal to $G^{\frac{1}{4}} \bmod \overline{\mathbb Q}^\times$ by Theorem \ref{sps}-(ii).
Therefore (\ref{conj4}) holds by (\ref{explicit}).
The cases $\tau=\tau_2,\tau_3,\tau_4$ can be proved similarly.
Moreover (\ref{explicit}), which is a special case of Conjecture \ref{maincnj}, 
implies the algebraicity of Stark units $\exp (-2\zeta_S'(0,\tau))$ with $\tau \in \mathrm{Gal}(K/F)$, 
as we show in Proposition \ref{mainprp}.
For example, let $\tau:=\mathrm{id} \in \mathrm{Gal}(K/F)$.
Since $K$ is the fixed subfield of $H_{\mathfrak f_0}$ under $\{\tau_1,\tau_3\}$,
we have $\exp (\zeta_S'(0,\mathrm{id}))=\exp (\zeta'(0,c_1))\exp (\zeta'(0,c_3)) \equiv \exp(X(c_1))\exp(X(c_3)) 
\equiv 1 \bmod \overline{\mathbb Q}^\times$
by Theorem \ref{main} and (\ref{explicit}).
\end{exm}

\begin{rmk}
Proposition \ref{mainprp} and Example \ref{exm} suggest that by studying the relation between CM-periods and the multiple gamma function, 
we will obtain partial results for Stark's conjecture. 
When $F=\mathbb Q$, we obtained a more precise result in \cite{Ka}.
We showed that Rohrlich's formula in \cite{Gr}, which is a special case of Conjecture \ref{YC}, 
and Coleman's formula in \cite{Co}, which is a $p$-adic analogue of Rohrlich's formula, 
imply not only the algebraicity but also a reciprocity law (\ref{rinka}) in the following sense: 
Let $K/F$ be $\mathbb Q(\zeta_m+\zeta_m^{-1})/\mathbb Q$ with $\zeta_m:=e^{\frac{2\pi i}{m}}$, $m\geq 3$.
We put $S:=\{$primes dividing $m\} \cup \{$the unique infinite place of $\mathbb Q \}$.
Then for $\tau \in \mathrm{Gal}(\mathbb Q(\zeta_m+\zeta_m^{-1})/\mathbb Q)$,
$\sigma \in \mathrm{Gal}(\overline{\mathbb Q}/\mathbb Q)$ we have 
\begin{equation} \label{rinka}
\begin{split}
&\exp(\zeta_S'(0,\tau))\in \overline{\mathbb Q}, \\
&\sigma(\exp(\zeta_S'(0,\tau)))\equiv \exp(\zeta_S'(0,\tau\sigma)) \bmod \mu_\infty,
\end{split}
\end{equation}
where we denote by $\mu_\infty$ the group of all roots of unity.
We note that (\ref{rinka}) is a part of Stark's conjecture in the case $F=\mathbb Q$. 
\end{rmk}

The opposite direction of Proposition \ref{mainprp} also holds, in the following sense.

\begin{prp} \label{mainprp2}
Let $F,K,K_{\mathrm{CM}},\mathfrak f_{K/F}| \mathfrak f_0$ be as in Conjecture \ref{maincnj}.
We put $K_{\mathrm{St}}$ to be the fixed subfield of $K$ by the complex conjugation $\rho_\mathrm{id}$ at $\mathrm{id} \in J_F$,
and $S:=\{$all places of $F$ dividing $\mathfrak f_0 \}$.
Then Conjecture \ref{YC} for $K_{\mathrm{CM}}/F$ 
and (\ref{wwsc}) for $K_{\mathrm{St}}/F$ imply Conjecture \ref{maincnj}.
\end{prp}

We note that, in the above Proposition, $K_{\mathrm{St}}$ is the maximal subfield of $K$
in which the real place corresponding to $\mathrm{id}$ splits completely, and that 
(\ref{wwsc}) is a weaker version of Conjecture \ref{wsc}.

\begin{proof}
Let $\mathrm{Art}_{\mathfrak f_0} \colon C_{\mathfrak f_0} \rightarrow \mathrm{Gal}(K/F)$ be as in Proposition \ref{mainprp}, 
$\mathrm{Art}_{\mathfrak f_0,\mathrm{CM}} \colon C_{\mathfrak f_0} \rightarrow \mathrm{Gal}(K_{\mathrm{CM}}/F)$ the associated map.
First we note that (\ref{wwsc}) for $K_{\mathrm{St}}/F$ states that 
\begin{equation} \label{wwsc2}
\prod_{c \in \mathrm{Art}_{\mathfrak f_0}^{-1}(\tau)}\exp(X(c)) 
\cdot \prod_{c \in \mathrm{Art}_{\mathfrak f_0}^{-1}(\tau \rho_{\mathrm{id}})}\exp(X(c))
\in \overline{\mathbb Q}^\times 
\qquad (\tau \in \mathrm{Gal}(K/F)).
\end{equation}
This equivalence follows from the same argument used in the proof of Proposition \ref{mainprp}.
Conjecture \ref{YC} states that for $\tau \in \mathrm{Gal}(K/F)$ we have
\begin{multline} \label{cnj3}
\prod_{c \in \mathrm{Art}_{\mathfrak f_0,\mathrm{CM}}^{-1}(\tau|_{K_{\mathrm{CM}}})}\exp(X(c)) \bmod \overline{\mathbb Q}^\times \\
=\pi^{[K:K_{\mathrm{CM}}]\zeta_{\mathfrak f_0}(0,\tau)}
p_{K_{\mathrm{CM}}}(\tau|_{{K_\mathrm{CM}}},
\mathrm{Inf}^{-1}(\sum_{\sigma \in G}\zeta_{\mathfrak f_0}(0,\sigma)\sigma))^{[K:K_{\mathrm{CM}}]},
\end{multline}
by Proposition \ref{trans} and (\ref{inf}).
On the other hand, we can write for $\tau_0 \in \mathrm{Gal}(K_{\mathrm{CM}}/F)$ 
\begin{equation*}
\prod_{c \in \mathrm{Art}_{\mathfrak f_0,\mathrm{CM}}^{-1}(\tau_0)}\exp(X(c)) 
=\prod_{\tau \in \mathrm{Gal}(K/F),\ \tau|_{K_\mathrm{CM}}=\tau_0}\prod_{c \in \mathrm{Art}_{\mathfrak f_0}^{-1}(\tau)}\exp(X(c)).
\end{equation*}
Hence in order to derive Conjecture \ref{maincnj} from (\ref{cnj3}), we need the following statement:
If $\tau,\tau' \in \mathrm{Gal}(K/F)$ satisfy $\tau|_{K_\mathrm{CM}}=\tau'|_{K_\mathrm{CM}}$, then we have 
\begin{equation} \label{need}
\prod_{c \in \mathrm{Art}_{\mathfrak f_0}^{-1}(\tau)}\exp(X(c)) 
\equiv \prod_{c \in \mathrm{Art}_{\mathfrak f_0}^{-1}(\tau')}\exp(X(c)) \bmod \overline{\mathbb Q}^\times.
\end{equation}
Indeed, for $\tau \in \mathrm{Gal}(K/F)$, $\iota,\iota' \in J_F$, we can write 
\begin{equation} \label{thmsc}
\prod_{c \in \mathrm{Art}_{\mathfrak f_0}^{-1}(\tau)}\exp(X(c)) 
\equiv \prod_{c \in \mathrm{Art}_{\mathfrak f_0}^{-1}(\tau \rho_\iota)}\exp(X(c))^{-1}
\equiv \prod_{c \in \mathrm{Art}_{\mathfrak f_0}^{-1}(\tau \rho_\iota\rho_{\iota'})}\exp(X(c)) \bmod \overline{\mathbb Q}^\times
\end{equation}
by Theorem \ref{main} (when $\iota,\iota' \neq \mathrm{id}$) and by (\ref{wwsc2}) (when $\iota,\iota' = \mathrm{id}$). 
By (\ref{kcm}), we see that (\ref{thmsc}) implies (\ref{need}).
Hence the assertion is clear.
\end{proof}


\begin{thebibliography}{99}

\bibitem[Co]{Co} Coleman, Robert F., On the Frobenius matrices of Fermat curves, $p$-adic analysis,
\textit{Lecture Notes in Math.} 
\textbf{1454}, Springer, Berlin, (1990), 173-193.

\bibitem[Da]{Da} Dasgupta, S., Stark's Conjectures (thesis).

\bibitem[Gr]{Gr} Gross, Benedict H., On the periods of abelian integrals and a formula of Chowla and Selberg (with an appendix by David E.\ Rohrlich),
\textit{Invent.\ Math.} 
\textbf{45} (1978), 193--211. 

\bibitem[Ka]{Ka} Kashio, T., Fermat curves and the reciprocity law on cyclotomic units (arXiv: 1502.04397),
to appear in \textit{J.\ Reine Angew.\ Math.},
the title had changed to ``Fermat curves and a refinement of the reciprocity law on cyclotomic units''.

\bibitem[KY1]{KY1} Kashio, T.\ and Yoshida, H., On $p$-adic absolute CM-Periods, I, 
\textit{Amer. J. Math.} 
\textbf{130}, (2008), no.\ 6, 1629-1685.

\bibitem[KY2]{KY2} Kashio, T.\ and Yoshida, H., On $p$-adic absolute CM-Periods, II, 
\textit{Publ.\ Res.\ Inst.\ Math.\ Sci.} 
\textbf{45}, (2009), no.\ 1, 187-225.

\bibitem[Shim]{Shim} Shimura, G., 
\textit{Abelian varieties with complex multiplication and modular functions},
Princeton Mathematical Series
\textbf{46} (1998), Princeton University Press.

\bibitem[Shin1]{Shin1} Shintani, T., On evaluation of zeta functions of totally real algebraic number fields at non-positive integers,
\textit{J.\ Fac.\ Sci.\ Univ.\ Tokyo Sect.\ IA Math.} 
\textbf{23} (1976), no. 2, 393-417.

\bibitem[Shin2]{Shin2} Shintani, T., On values at $s=1$ of certain $L$ functions of totally real algebraic number fields, 
\textit{Algebraic Number Theory, Proc.\ International Sympos., Kyoto, 1976}, 
Kinokuniya, Tokyo, 1977, 201-212.

\bibitem[Yo]{Yo} Yoshida, H., 
\textit{Absolute CM-Periods},
Mathematical Surveys and Monographs, 
\textbf{106}, Amer.\ Math.\ Soc., Providence, RI, 2003.

\end{thebibliography}
\end{document}